\documentclass[11pt]{article} 

\usepackage{epsfig} 
\usepackage{amssymb,amsmath,amsthm,amscd}
\usepackage{latexsym} 
\usepackage{cancel}
\usepackage{xcolor}
\usepackage{amssymb}

\usepackage[margin=1.1in]{geometry}
\usepackage{soul}
 
\newtheorem{thm}{Theorem}[section] 
\newtheorem{prop}[thm]{Proposition} 
\newtheorem{cor}[thm]{Corollary} 
\newtheorem{lem}[thm]{Lemma}

\theoremstyle{definition}

\theoremstyle{remark} 
\newtheorem{rem}{Remark}[section]  
\def\eqref#1{(\ref{#1})} 

\newcommand{\bbR}{{\mathbb{R}}}

\newcommand {\mat}      [1] {\left[\begin{array}{#1}}
\newcommand {\rix}          {\end{array}\right]}
\newcommand {\de}      [1] {\left|\begin{array}{#1}}
\newcommand {\nt}          {\end{array}\right|}

\newcommand{\bstar}       {\begin{eqnarray*}}
\newcommand{\estar}       {\end{eqnarray*}}
\newcommand{\eqn}       {\begin{eqnarray}}
\newcommand{\enn}       {\end{eqnarray}}
\newcommand{\eq}[1]   {\begin{equation}\label{#1}}
\newcommand{\en}      {\end{equation}}

\begin{document}  
\begin{titlepage}  
\title{Reversal potential  and reversal permanent charge with unequal diffusion coefficients via classical Poisson-Nernst-Planck models}  
\author{Hamid Mofidi\footnote{Department of Mathematics, University of Kansas, 
1460 Jayhawk Blvd., Room 405,
Lawrence, Kansas 66045, USA ({\tt h.mofidi@ku.edu}).}\; and Weishi Liu\footnote{Department of Mathematics, University of Kansas, 
1460 Jayhawk Blvd., Room 405,
Lawrence, Kansas 66045, USA ({\tt wsliu@ku.edu}).  
}}
\date{\today} 
\end{titlepage}

\maketitle    

\begin{abstract}  In this paper, based on   geometric singular perturbation  analysis of a quasi-one dimensional  Poisson-Nernst-Planck model for ionic flows, we study the problem of zero current condition for ionic flows through membrane channels with  a simple profile of permanent charges. 
 For ionic mixtures of multiple ion species, under equal diffusion constant condition, Eisenberg, et al [{\em Nonlinearity {\bf 28} (2015), 103-128}] derived a system of two equations for determining the reversal potential and an equation for the reversal permanent charge.  The equal diffusion constant condition is significantly degenerate from physical points of view. For unequal diffusion coefficients, the analysis becomes extremely challenging. This work will focus only on   two ion species, one positively charged (cation) and one negatively charged (anion), with two arbitrary diffusion coefficients.  Dependence of reversal potential on channel geometry and diffusion coefficients has been investigated experimentally, numerically, and  analytically in simple setups,  in many works. In this paper, we identify two governing equations for the zero current, which enable one to mathematically analyze how the reversal potential depends on the channel structure and diffusion coefficients.  In particular, we are able to show, with a number of concrete results, that the possible different diffusion constants indeed make significant differences. The inclusion of channel structures is also far beyond the situation  where the Goldman-Hodgkin-Katz (GHK) equation might be applicable. A comparison  of our result with the GHK equation is provided.   The dual problem of reversal permanent charges is briefly discussed too.
 \end{abstract} 


\section{Introduction.} 
\setcounter{equation}{0}
Ion channels,  proteins embedded in membranes,  provide a major pathway for cells to communicate with each other and with the outside to transform signals and to conduct group tasks (\cite{BNVHEG, Eis00, Hil01, Hille89}).
 The essential structure of an ion channel is its shape and its permanent charge.
 The shape of a typical channel could be approximated as a cylindrical-like domain.
Within an ion channel, amino acid side chains are distributed mainly over a ``short" and ``narrow" portion of the channel, with acidic side chains contributing negative charges and basic side chains providing positive charges. It is specific of side-chain distributions, which is referred to as the permanent charge of the ion channel.  The function of channel structures is to select the types of ions and to facilitate the diffusion of ions across cell membranes. 
 
At present, these permeation and selectivity properties of an ion channel are mainly derived from the I-V relation measured experimentally (\cite{Hil01, Hille89, H51, HH52a, HH52b, HH52c}). 
  Individual fluxes carry more information than the I-V relation. However, it is expensive and challenging to measure them (\cite{HK55, JEL17}). 
The I-V relation is a reasonable response of the channel structure on ionic fluxes, but it depends on boundary conditions that are, in fact, driving forces of ionic transport. 
The multi-scale feature of the problem with multiple physical parameters enables the system to have high flexibility and to show rich phenomena/behaviors -- a great advantage of ``natural devices" (see, e.g., \cite{BKSA09, Eis}). On the other hand, the same multi-scale characteristic with multiple physical parameters presents a remarkably demanding task for anyone to derive meaningful information from experimental data, also given the fact that the internal dynamics cannot be discerned with the present technique.

Mathematical analysis plays essential and unique roles for revealing mechanisms of observed biological phenomena and for discovering new ones, assuming a more or less explicit solution of the associated mathematical model can be achieved. The latter is often too much to hope. Nonetheless, there have been some accomplishments lately in analyzing Poisson-Nernst-Planck (PNP) models for ionic flows through ion channels (\cite{EL07, ELX15, JEL17, JLZ15, Liu05, Liu09, LX15,PJ}, etc.). 

There are many models, from low resolution to high, for ionic flows in various settings. One can derive PNP systems as reduced models from molecular dynamic models, Boltzmann equations, and variational principles (\cite{Bar,HEL10,HFEL12,SNE01}). The PNP type models have different levels of resolutions: the classical PNP treats dilute ionic mixtures, so ions are approximated by point-charge (no ion-to-ion interactions are included). The PNP-HS takes into consideration of volume exclusive by treating ions as hard-spheres (but the charges are located at the center), etc..  More sophisticated models, such as coupling PNP and Navier-Stokes equations for aqueous motions, have also been revealed (see, e.g. \cite{BKSA09, CEJS95, EHL10, SR81}). 


 Focusing  on certain  critical characteristics of the biological systems, PNP models serve as suitable ones for analysis and numerical simulations of ionic flows. In this work, we apply the classical PNP model and consider a cylinder-like channel to center the basic comprehension of possible effects of general diffusion coefficients in ionic channels. One cannot achieve gating and selectivity by a simple classical PNP model as it treats ions as point charges. However, the primary finding on reversal potentials and their dependence on permanent charges and ratios of diffusion constants seems essential, and some are non-intuitive and worthy of further studies.  At a later time, one should examine more fundamental detail and more correlations between ions in PNP models such as those including various potentials for ion-to-ion interaction accounting for ion sizes effects and voids  \cite{JL12, SL18}. 
 
We are interested in {\em reversal potentials} (or Nernst potentials) as well as {\em reversal   permanent charges}.  They are defined by zero total current: for fixed other physical quantities, the total current $I=I(V,Q)$ depends on the transmembrane potential $V$ and the permanent charge $Q$. For fixed $Q$, {\em a reversal potential} $V=V_{rev}(Q)$ is a transmembrane potential that produces zero current $ I(V_{rev}(Q),Q)=0$. Likewise, for fixed transmembrane potential $V$, {\em a reversal permanent charge} $Q=Q_{rev}(V)$  is a permanent charge that produces zero current $ I(V, Q_{rev}(V))=0$.
 
Nernst was among the first who considered reversal potential and,  for one ion species case, formulated an equation -- now called  the {\em Nernst equation} -- for the reversal potential. Following a treatment of Mott  for electronic conduction in the copper-copper oxide rectifier (\cite{Mott39}),   the Nernst equation was generalized by Goldman (\cite{Goldman}), and Hodgkin and  Katz (\cite{HK49}) -- called {\em Goldman-Hodgkin-Katz (GHK) equation} - for reversal potentials involving multiple ion species. {  The derivations were based on an inaccurate assumption (maybe for simplicity or to be feasible) that the electric potential $\phi(x)$ is linear in $x$ -- the coordinate along the longitude of the channel. Unfortunately, there was no substitute yet for their equations. }

 Recently in \cite{ELX15}, the authors investigated the problem of determining reversal potentials and reversal permanent charges based on rigorous analysis on the Poisson-Nernst-Planck models. For the case when {\em all diffusion constants are equal}, the results are very complete. In particular, a system of two equations is derived that will lead to a determination of the reversal potential, and one equation is derived for the reversal permanent charge. 
 On the other hand, {\em the equal diffusion constants case is degenerate}, which is known from biological point of view even for ionic mixtures of two ion species.
 
 The case with unequal diffusion coefficients has been studied in many works. We will mention a few here and refer the reader to the references therein for more works.
 In \cite{BK18},  the autors conducted a perturbation  analysis from a special solution, {\em a time independent and spatially homogeneous equilibrium solution},   with the ratio $\epsilon=VF/RT$ of a weak applied voltage $V$ and the thermal voltage $RT/F$ as the perturbation parameter. Based on information obtained from the $O(\epsilon)$ terms, the authors identified two time scales of the dynamics: a time scale for charging and a time scale for diffusion. Most importantly, for equal diffusion coefficient, the diffusion process for $O(\varepsilon)$ terms does not occur -- an important effect of unequal ionic mobilities. In the review article \cite{BKSA09}, among many basic topics of electrodiffusion processes,  the authors addressed an important aspect of mobilities, in our opinion,  that is,  how mobilities as well as their spatial inhomogeneities are influenced by other parameters.     
  In the paper \cite{HBRR18},  motivated by several analyses on complications of nonlinear electrodiffusion models with equal ionic mobilities of cation and anion,  the authors  examine the cases with unequal mobilities by  computations of a fully nonlinear electrokinetic model and observed the appearance of a long-range steady field due to   unequal mobilities.

  In this work, allowing different diffusion coefficients, we start our investigation on reversal potentials and reversal permanent charges  for two ion species. We are particularly interested in the effect of unequal diffusion coefficients on the properties of reversal potentials and reversal permanent charges.
 
  The geometric singular perturbation (GSP) framework developed in \cite{EL07, Liu05,Liu09} particularly for analyzing PNP models for ionic flow is again applied as in \cite{ELX15} to get a system of algebraic equations for the problem. The solution method of solving/analyzing the algebraic system is simply different from that in \cite{ELX15} due to the difference between $D_1$ and $D_2$. The difficulty is overwhelmingly increased -- more than technical. An important step in our analysis is a reduction of the algebraic system to two nonlinear equations that turns out to work effectively. 
As a consequence, this reduced system  allows one to, for the first time, examine  how the reversal potential depends on the channel structure, boundary concentrations and diffusion coefficients.    In particular, we are able to establish a number of precise differences that   possible different diffusion constants  make. Some of these results can be explained qualitatively  in terms of physical
 intuitions, for examples,  the dependence of the sign of reversal potential on interplay between diffusion constants, boundary conditions and permanent charge (Theorem \ref{VonQ} and Corollary \ref{V_rev-Sgn}), how the monotonicity of the reversal potential in the permanent charge depends on the relative sizes of the diffusion constants together with the boundary conditions (Theorem \ref{monoVinQ}), etc.    Some are counterintuitive, including the specific dependence on the boundary concentrations of the monotonicity of the reversal potential in $\theta=(D_2-D_1)/{(D_2+D_1)}$ (Proposition \ref{Vonp} and Remark \ref{remVonp}). 
All these results are not known before and there are also several concrete open questions that we share our belief but could not verify.
  The well-known  GHK  equation for the reversal potential is briefly discussed and a short comparison with our result is provided.


  The highlights of our studies in this manuscript are as follows:

	\begin{itemize}
		\item[a.] A mathematically derived system for the zero-current condition (see System (3.4)) that we employed to determine the zero-current flux and  the reversal potential in terms of other parameters (see Equations (3.8) and (4.1));
		
		\item[b.] an examination on how the reversal potential depends on permanent charge and diffusion constants (see Section 4);
		
		\item[c.] a comparison between our reversal
		potential and that of GHK equation in the particular setting (see Section 4.3).
		
	\end{itemize}
	
	Besides, there are some qualitatively relevant but non-intuitive outcomes presented in this work that may help to guide experimentation, and some might not be obvious in intuitive reasoning about ion channel operation.


 The rest of paper is divided as follows. In Section \ref{sec-PNP} we introduce the problem and provide the basic setup for our problem in Section \ref{setup}.  We apply  the  GSP theory in Section \ref{sec-GSP} to derive  the matching system of algebraic equations for the zero current condition.     In Section \ref{sec-ReducedSys}, we discuss the reduced system   for a more straightforward case and make preparation prepare the stage      for our main concern.   The topics on reversal potential, its existence, uniqueness and dependence on permanent charge and diffusion coefficients, are analyzed in Section \ref{revP}.   The topic on reversal permanent charge is briefly discussed in Section \ref{revPC}. Section \ref{ConSec} is a short conclusion. The appendix (Section \ref{Appendix}) details the reduction to the system of two equations for the zero current.


\subsection{A Quasi-one-dimensional PNP Model for Ion Transports.}\label{sec-PNP}
\setcounter{equation}{0}
The PNP system has been extensively studied by simulations and computations  (\cite{AEL, Bar, BCE, BCEJ, CE, ChK, HCE, IR}).   It is clear from these simulations that macroscopic reservoirs -- mathematically boundary conditions -- must be included in the mathematical formulation to describe the actual behavior of channels (\cite{GNE, NCE}).
On the basis that ion channels have narrow cross-sections relative to their lengths, 3-D PNP type models  are further reduced to quasi-one-dimensional models (\cite{LW,NE}):
 \begin{align}\label{PNP-dim}
 \begin{split} 
{\text{Poisson:} }\quad &\frac{1}{A(X)}\frac{d}{dX}\Big(\varepsilon_r(X)\varepsilon_0A(X)\frac{d}{dX}\Phi\Big)=-e_0\Big(\sum_{s=1}^nz_sC_s+{\cal Q}(X)\Big),\\
\text{Nernst-Planck:} \quad & \frac{d}{dX}{\cal J}_k=0, \quad -{\cal J}_k=\frac{1}{k_BT} {\cal D}_k(X) A(X)C_k\frac{d}{d X}\mu_k, \quad
 k=1,2,\cdots, n
\end{split} 
\end{align} 
where $X\in  [0,L]$ is the coordinate along the longitudinal axis of the channel, $A(X)$ is the 
 area of cross-section  of the channel over the location $X$; ${\cal Q}(X)$ is the permanent charge density, $\varepsilon_r(X)$ is the  relative dielectric coefficient, $\varepsilon_0$ is the vacuum permittivity, $e_0$ is the elementary charge, $k_B$ is the Boltzmann constant, $T$ is the absolute temperature;  $\Phi$ is the electric potential,   and, 
 for  the $k$th ion species, $C_k$ is the
concentration, $z_k$ is  the valence (the number of charges per particle), $\mu_k$ is the electrochemical potential depending on $\Phi$ and $C_k$,   ${\cal J}_k(X)$ is the flux density through the cross-section over $X$, and ${\cal D}_k(X)$ is  the diffusion coefficient.

 Equipped with system (\ref{PNP-dim}),  we impose the following   boundary conditions \cite{GE02}, for $k=1,2,\cdots, n$,
\begin{equation}
\Phi(0)={\cal V}, \ \ C_k(0)=L_k>0; \quad \Phi(l)=0,  \ \
C_k(l)=R_k>0. \label{ssBV}
\end{equation}
 
 For an analysis of the boundary value problem (\ref{PNP-dim}) and (\ref{ssBV}), we will work on a dimensionless form. Let $C_0$ be a characteristic concentration of the problems, for example,
 \[C_0=\max_{1\le k\le n}\big\{ L_k, R_k, \sup_{X\in[0,L]}|{\cal Q}(X)|\big\}.\]
 Set \[{\cal D}_0=\max_{1\le k\le n}\{\sup_{X\in [0,L]}{\cal D}_k(X)\}\;\mbox{ and }\; \bar{\varepsilon}_r=\sup_{X\in [0,L]} \varepsilon_r(X).\]
Let
\begin{align}\label{dimless}\begin{split}
&\varepsilon^2=\frac{\bar{\varepsilon}_r\varepsilon_0k_BT}{L^2e_0^2C_0},\quad
  \hat{\varepsilon}_r(x)=\frac{\varepsilon_r(X)}{\bar{\varepsilon}_r},\quad x=\frac{X}{L},\quad  h(x)=\frac{A(X)}{L^2},\quad D_k(x)=\frac{{\cal D}_k(X)}{{\cal D}_0},\\
  & Q(x)=\frac{{\cal Q}(X)}{C_0},  \quad \phi(x)=\frac{e_0}{k_BT}\Phi(X), \quad c_k(x)=\frac{C_k(X)}{C_0},\quad \hat{\mu}_k=\frac{1}{k_BT}\mu_k, \quad J_k=\frac{{\cal J}_k}{LC_0 {\cal D}_0}. 
  \end{split}
 \end{align}
In terms of the new variables,     BVP (\ref{PNP-dim}) and (\ref{ssBV}) becomes, for $k=1,2,\cdots,n$,
\begin{align}\label{PNP}
\begin{split} 
 \frac{\varepsilon^2}{ h(x)}  \frac{d}{dx}\left(\hat{\varepsilon}_r(x)h (x)\frac{d}{dx}\phi\right)=&- \sum_{s=1}^nz_s c_s -  Q(x), \\ 
\frac{d J_k}{dx}  =0, \quad  -J_k=&h (x)D_k(x)c_k\frac{d  }{dx}\hat{\mu}_k ,   
\end{split} 
\end{align} 
with the boundary conditions  
\begin{equation}\label{BV} 
\phi(0)=V=\frac{e_0}{k_BT}{\cal V},  \quad c_k(0)=l_k=\frac{L_k}{C_0}; \quad \phi(1)=0,\quad c_k(1)=r_k=\frac{R_k}{C_0}.
\end{equation} 

The dimensionless parameter $\varepsilon$ is the ratio of Debye length $\lambda_D$ over the distance $L$ between the two applied electrodes, that is,
$
\varepsilon =  {\lambda_D}/{L}.
$
We will assume     $\varepsilon$ is small, which allows us to treat the problem as a singularly perturbed problem.   
The dimensionless parameter $\varepsilon$ may not be small in general but, for ion channel problems, it is typically small.
For example,  if   $L= 2.5\mbox{ nm} =2.5 \times 10^{-9} \mbox{ m}$ and $C_0 = 10 \mbox{ M}$, then it is shown in \cite{EL17,MEL20} that $\varepsilon \approx 10^{-3}$.

 We impose the electroneutrality conditions on the concentrations to avoid sharp boundary layers, which cause significant changes (large gradients) of the electric potential and concentrations near the boundaries so that computation of these values has non-trivial uncertainties. 
	
 \begin{align}\label{neutral}
 \sum_{s=1}^nz_sl_s= \sum_{s=1}^nz_sr_s=0.
 \end{align}
 
 The electrochemical potential $\hat{\mu}_k(x)=\hat{\mu}_k^{id}(x)+\hat{\mu}_k^{ex}(x)
$ for the $k$th ion species
consists of the ideal component $\hat{\mu}_k^{id}(x)$ and the excess
component $\hat{\mu}_k^{ex}(x)$,
where the ideal component is 
\begin{equation}
\hat{\mu}_k^{id}(x)=z_k \phi(x)+ \ln c_k(x). \label{Ideal}
\end{equation}
The   classical  PNP model    only deals with  the ideal component $\hat{\mu}_k^{id}(x)$, which ignores the size of ions and reflects the entropy of the dilute ions in water. Dilute solutions tend to approach ideality as they proceed toward infinite dilution  (\cite{RC13}).
  This component is essential for dealing with properties of crowded ionic mixtures where concentrations exceed say 1M.
 
 For given $V$, $Q(x)$, $l_k$'s and $r_k$'s,   if  $(\phi(x;\varepsilon), c_k(x;\varepsilon), J_k(\varepsilon))$ is a solution of the boundary value problem (\ref{PNP}) and (\ref{BV}), then  the current ${\cal I} $ is   
 \begin{align}\label{IV}
{\cal I}= {\cal I}(\varepsilon)=\sum_{s=1}^nz_sJ_s(\varepsilon).
 \end{align}
 
 We will be interested in the zero order approximation of $I= {\cal I}(0)$ and $J_k=J_k(0)$.
Note that, $J_k$ depends on $V$, $Q(x)$, $l_k$'s and $r_k$'s, so is $I$.  
As mentioned before, we will focus mainly on the dependance of   $I=I(V,Q)$ on the electric potential $V$ and permanent charge $Q$. Particularly, for fixed $Q$, the electric potential $V$ so that $I(V,Q)=0$ is {\em the reversal potential}. The reversal potential has been used to identify the type (i.e., selectivity) of ion channels in biological experiments since 1949 (\cite{HHK49, HK49}).  Similarly, for fixed $V$, the permanent charge $Q$ that makes $I(V,Q)=0$ is called {\em a reversal permanent charge} as introduced in \cite{ELX15}.  For the existence of a reversal permanent charge $Q$ of a general form, a necessary condition is that
{\em the quantities $z_k(z_kV+\ln l_k-\ln{ r_k})$, for $k=1,2,\ldots, n$, cannot have the same sign} (Proposition 1.1 in \cite{ELX15}).

%
%
%
%


\subsection{Setup of the Case Study for $n=2$.}\label{setup}
We now specify the case treated in this paper. We will examine the question    by working on the simplest model,  {\em the classical PNP} model (\ref{PNP}) with the ideal electrochemical potential $\hat{\mu}_k=z_k\phi+\ln c_k$, a simple profile of a permanent charge $Q(x)$ (see (A2) below), and the boundary condition (\ref{BV}). 
 We will  focus on the case of two ($n=2$) ion species but allow different diffusion coefficients. More precisely, we will assume
{\em \begin{itemize}
\item[(A0)] $\hat{\varepsilon}(x)=1$ and, for $k=1,2$,  $D_k(x)=D_k$ is a constant;
\item[(A1)] Electroneutrality boundary conditions (\ref{neutral});  
\item[(A2)] A piecewise constant permanent charge $Q$ with one nonzero region; that is, for a partition $0=x_0<x_1<x_2<x_3=1$ of   $[0,1]$, 
 \begin{align}\label{Q}
 Q(x)=\left\{\begin{array}{ll}
 Q_1=Q_3=0, & x\in (x_0,x_1)\cup (x_2,x_3),\\
 { Q_2=2Q_0}, & x\in (x_1,x_2),
 \end{array}\right.
 \end{align}   
 where  $Q_2=2Q_0$ is an arbitrary constant.
 \end{itemize}}
 
For permanent charges $Q$  of the form  in (\ref{Q}) and for general $n$, under the condition of equal diffusion coefficients $D_k$'s, the topics on the reversal potential and reversal permanent charges were examined completely in \cite{ELX15}. It turns out that the condition of equal diffusion coefficients is highly degenerate (see  Remark \ref{slowphi}).  This is the main technical reason for us to limit to the case $n=2$ in this work.
  
  \begin{figure}[h]\label{Fig-PermChargeHx}
	\centerline{\epsfxsize=5.5in \epsfbox{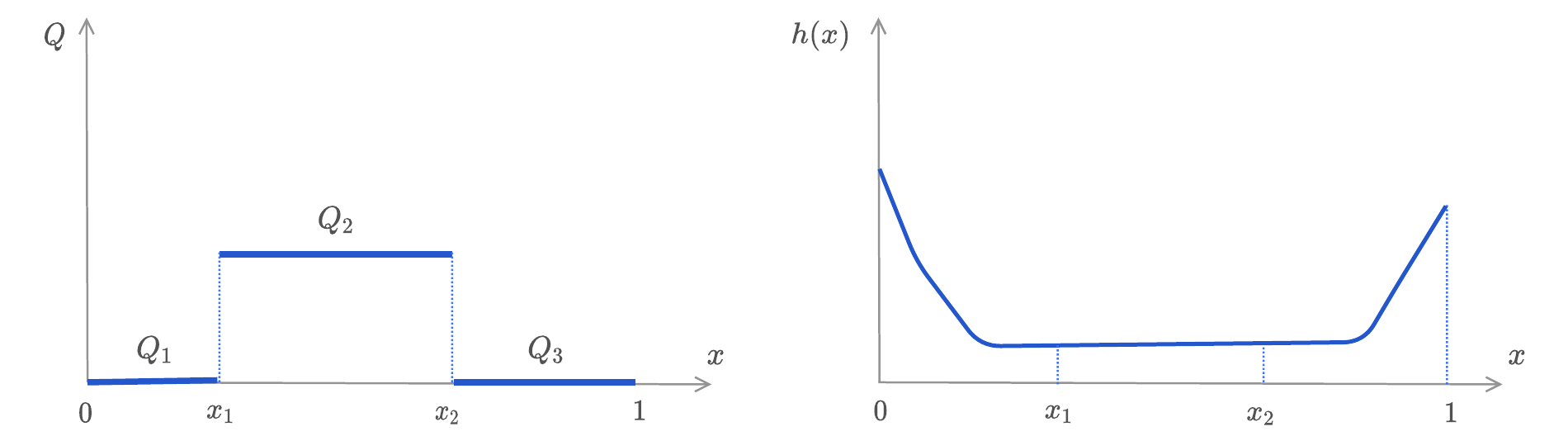} }\vspace*{-.1in}
	\caption{\em  The left panel shows the permanent charge; the right panel shows the $h(x)$ where the neck of channel is between $x_1$ and $x_2$, within which permanent charge is confined. }
	\label{Fig-PermChargeHx}
\end{figure}

   \section{GSP for  the BVP (\ref{PNP}) and (\ref{BV}) and the Results on Current Reversal for the Case Study with $n=2$. } \label{sec-GSP}
  \setcounter{equation}{0}
  In this section, for two ion species, we will apply a GSP to reduce the BVP  (\ref{PNP}) and (\ref{BV}) with current $I=0$ to a system of algebraic equations (\ref{Matching}) and (\ref{Matching2}). 
  
 In \cite{Liu09},  a GSP framework, combining with special structures of PNP systems, has been developed for studying  the BVP (\ref{PNP}) and (\ref{BV}). This general dynamical system framework and the subsequent analysis have demonstrated the great power of analyzing  PNP type problems with potential boundary and internal layers (see \cite{EL07, Liu05,  Liu09, LX15} for study on  the classical PNP models, \cite{LLYZ} for PNP with a local excess hard-sphere components, and \cite{JL12,  LTZ, SL18} for PNP with nonlocal excess hard-sphere components). 

For convenience, we will give a brief account of the relevant results in \cite{Liu09} (with slightly different notations) and refer the readers to  the paper for details. We remind the readers that  we  will work on the classical PNP  with ideal electrochemical potential $\mu_k=z_k\phi+\ln c_k$.


{ In the following, we consider the case with $n=2$ ion species.   As we go further, over the slow layers, we apply another limitation on the valences of ion species and consider the case where $z_1 =- z_2$.}

Denote the derivative with respect to $x$ by  overdot  and
introduce $u=\varepsilon \dot \phi$  and $w=x$.
  System~\eqref{PNP} becomes, for $k=1,2$,   
 \begin{align}\label{slow}\begin{split} 
  \varepsilon\dot \phi=&   u,\quad
  \varepsilon \dot u=  -z_1c_1-z_2c_2-Q(w) - \varepsilon \frac{h_w(w)}{h(w)} u, \\
  \varepsilon \dot c_k=&-z_kc_ku - \varepsilon  \frac{J_k}{D_k h(w)},  \quad \dot J_k= 0,\quad \dot w=1.
 \end{split}
 \end{align}
  System~\eqref{slow} will be treated as a dynamical system with the phase space
$\bbR^{7}$ and the independent variable $x$ is viewed as time for the dynamical system.

The introduction of the new state variable $w = x$ and the augmentation of the equation $\dot w=1$ is crucial for two reasons: first of all, it makes system~\eqref{slow} an autonomous system which will be treated as a dynamical system; secondly,  one can then covert the boundary value problem from $x=0$ to $x=1$ to a connecting problem between $B_0$ and $B_3$, which is stated below display (\ref{orb}). Note that $w=x=0$ is encoded in $B_0$ and $w=x=1$ is encoded in $B_3$. Thus, whenever one finds an orbit between $B_0$ and $B_3$, it automatically starts from $x=0$ and ends at $x=1$. In particular, one can multiply the vector field of system~\eqref{slow} by any (positive) scale function which may depend on state variables since this would not change the phase space portrait of system~\eqref{slow}. The latter is applied in this paper that transforms (\ref{Dslow1}) to (\ref{DSlow-y}). The need for the change of variable is better shown in [Liu09] where more than two ion species are involved.

The boundary condition~\eqref{BV}  becomes,  for $k=1,2$, 
 \begin{align}\label{BV4cp}
 \phi(0)=V, \; c_k(0)= l_k, \; w(0)=0;\quad
\phi(1)=0, \; c_k(1)=r_k, \; w(1)=1. 
\end{align}

Following the framework in \cite{Liu09}, we convert the boundary value problem to a connecting problem. To this end, 
we denote $C=(c_1,  c_2)^T$ and $J=(J_1, J_2)^T$, and
we preassign values of $\phi$ and $C$ at $x_1$ and $x_2$:
\[\phi(x_j)=\phi^{[j]}\;\mbox{ and }\; C(x_j)=C^{[j]}\;\mbox{ for }\; j=1,2.\]
Now for  $j=0,1,2,3$, let $B_j$ be the subsets of the phase space $\bbR^{7}$ defined by
\begin{align}\label{orb}
B_j=\Big\{(\phi, u, C, J, w):\;\phi=\phi^{[j]},\; C=C^{[j]},\; w=x_j\Big\},
\end{align}
 
Note that the set $B_0$ is associated to the boundary condition in (\ref{BV4cp}) at $x=0$ and the set $B_3$ is associated to the boundary conditions at $x=1$.
 Thus, the  BVP (\ref{PNP}) and (\ref{BV}) is equivalent to the following 
{\em connecting orbit problem}: finding an orbit of (\ref{slow}) from $B_0$ to $B_3$. 
The construction would be accomplished by finding first a singular connecting orbit -- a union of limiting slow orbits and limiting fast orbits,
and then applying the exchange lemma to show the existence of a connecting orbit for $\varepsilon>0$ small (see \cite{Liu09} for details). For the problem at hand, the construction of a singular orbit consists of one singular connecting orbit from $B_0$ to $B_1$, one from $B_1$ to $B_2$, and one from $B_2$ to $B_3$ with a matching of $(J_1,J_2)$ and $u$ at $x_1$ and $x_2$ (see again \cite{Liu09} for details).

\subsection{Singular connecting orbits from $B_{j-1}$ to $B_j$ for $j=1,2,3$}
By setting $\varepsilon=0$ in system (\ref{slow}), we  get the  
 {\em slow manifold} 
 \[{\mathcal Z}_j=\Big\{u=0,\; z_1c_1+z_2c_2+ Q_j=0\Big\}.\] 
  In terms of the  independent variable $\xi= x/\varepsilon$, 
  we obtain {\em the fast system} of~\eqref{slow}, for $k=1,2$, 
 \begin{align}\label{fast}\begin{split} 
      & \phi'=   u,\;
    u'=  -z_1c_1-z_2c_2 - Q_j- \varepsilon \dfrac{h_w(w)}{h(w)} u ,\\
   & c_k'=-z_kc_ku -\varepsilon \dfrac{J_k}{D_k h(w)}, \quad J'= 0, \quad  w'=\varepsilon,   
 \end{split}
 \end{align}
where   prime denotes the derivative with respect to   $\xi$.
 The limiting fast system is, for $k=1,2$,
\begin{align}\label{limfast}\begin{split} 
     \phi'=&   u,\quad
    u'=  -z_1c_1-z_2c_2  - Q_j, \quad
    c_k'=-z_kc_ku, \quad J'=0,\quad  w'=0.
 \end{split}
 \end{align}

The slow manifold ${\mathcal Z}_j$ is precisely the set of equilibria of (\ref{limfast}) with   $\dim{\mathcal Z}_j=5$.  For the linearization of (\ref{limfast}) at each point on ${\mathcal Z}_j$,  there are  $5$ zero eigenvalues associated to the tangent space of ${\mathcal Z}_j$ and the other two eigenvalues 
 are $ \pm\sqrt{z_1^2c_1+z_2^2c_2}$.
Thus, ${\mathcal Z}_j$ is normally hyperbolic (see~\cite{Fen, HPS}). We will denote the stable and unstable manifolds of ${\mathcal Z}_j$ by
$W^s({\mathcal Z}_j)$ and  $W^u({\mathcal Z}_j)$, respectively.  


Let $M^{[j-1,+]}$   be the collection of all forward orbits from  $B_{j-1}$ under the flow of \eqref{limfast}
and let  $M^{[j,-]}$   be the collection of all backward orbits from $B_{j}$.  Then the set of forward orbits from $B_{j-1}$ to 
${\mathcal Z}_j$ is  $N^{[j-1,+]}=M^{[j-1,+]}\cap W^s({\mathcal Z}_j)$, and the set of backward orbits from $B_{j}$ to 
${\mathcal Z}_j$ is $N^{[j,-]}=M^{[j,-]}\cap W^u({\mathcal Z}_j)$. Therefore, the singular layer $\Gamma^{[j-1,+]}$ at $x_{j-1}$ satisfies  $\Gamma^{[j-1,+]}\subset  N^{[j-1,+]} $ and the singular layer $\Gamma^{[j,-]}$ at $x_{j}$ satisfies $\Gamma^{[j,-]}\subset N^{[j,-]}$.
All those important geometric objects  are {\em explicitly} characterized in \cite{Liu09}.

  \subsubsection{Fast (layer) dynamics for singular layers  at $x_1$ and $x_2$.}\label{inner}
  
The limiting fast (layer) dynamics conserve electrochemical potentials, and hence, do not depend on diffusion constant (see, e.g. Proposition 3.3 in \cite{Liu09}). We thus can apply the result about the fast dynamics from \cite{ELX15} directly and only point out the differences.  The relevant results are Lemmas 3.4 and 3.5 in \cite{ELX15}. The differences are that we have to keep $\phi^{[1,-]}$, $\phi^{[1,+]}$, $\phi^{[2,-]}$ and $\phi^{[2,+]}$ here in  this paper, while in \cite{ELX15} it is known that $\phi^{[1,-]}=V$,
  $\phi^{[1,+]}=\phi^{[2,-]}$ (denoted by ${\cal V}^*$ there) and $\phi^{[2,+]}=0$. With this modification, these lemmas are cast below for $n=2$.
 
  \begin{lem}\label{fast1}
  The fast layer dynamics  over $x=x_1$ provides, for $k=1,2$,
  \begin{itemize}
  \item[(i)] relative to $(0,x_1)$ where $Q_1=0$,
 \[  z_1c_1^{[1]}e^{z_1(\phi^{[1]}-\phi^{[1,-]})}+ z_2c_2^{[1]}e^{z_2(\phi^{[1]}-\phi^{[1,-]})}=0,
 \quad c_k^{[1,-]}=c_k^{[1]}e^{z_k( \phi^{[1]}-\phi^{[1,-]})};\]
 \item[(ii)] relative to $(x_1,x_2)$ where $Q_2\neq 0$,
 \[  z_1c_1^{[1]}e^{z_1(\phi^{[1]}-\phi^{[1,+]})}+ z_2c_2^{[1]}e^{z_2(\phi^{[1]}-\phi^{[1,+]})} + Q_2=0, \quad c_k^{[1,+]}=c_k^{[1]}e^{z_k( \phi^{[1]}-\phi^{[1,+]})};\]
 \item[(iii)] the matching   $u_-^{[1]}=u_+^{[1]}$: \quad
 $c_1^{[1,-]} + c_2^{[1,-]} =c_1^{[1,+]} + c_2^{[1,+]} +Q_2(\phi^{[1]}-\phi^{[1,+]})$.
 \end{itemize}
 \end{lem}
 
  \begin{lem}\label{fast2}
  The fast layer dynamics  over $x=x_2$ provides, for $k=1,2$,
  \begin{itemize}
  \item[(i)] relative to $(x_1,x_2)$ where $Q_2\neq 0$,
 \[  z_1c_1^{[2]}e^{z_1(\phi^{[2]}-\phi^{[2,-]})}+ z_2c_2^{[2]}e^{z_2(\phi^{[2]}-\phi^{[2,-]})} + Q_2=0,
 \quad c_k^{[2,-]}=c_k^{[2]}e^{z_k( \phi^{[2]}-\phi^{[2,-]})};\]
 \item[(ii)] relative to $(x_2,1)$ where $Q_3= 0$ , 
 \[  z_1c_1^{[2]}e^{z_1(\phi^{[2]}-\phi^{[2,+]})}+ z_2c_2^{[2]}e^{z_2(\phi^{[2]}-\phi^{[2,+]})} =0, \quad c_k^{[2,+]}=c_k^{[2]}e^{z_k( \phi^{[2]}-\phi^{[2,+]})};\]
 \item[(iii)] the matching   $u_-^{[2]}=u_+^{[2]}$: \quad
 $c_1^{[2,-]} + c_2^{[2,-]} +Q_2(\phi^{[2]}-\phi^{[2,-]}) =c_1^{[2,+]} + c_2^{[2,+]} $.
 \end{itemize}
 \end{lem} 

 \subsubsection{Slow dynamics for regular layers over $(x_{j-1},x_j)$.}\label{outer}
 The degeneracy of equal diffusion coefficients shows in the slow dynamics. We will point out the exact place in the following construction of the slow orbits over the slow manifold
 \[{\cal Z}_j=\Big\{u=0, \;z_1c_1+z_2c_2 +Q_j=0\Big\}.\]
Note that  system \eqref{slow} is degenerate at $\varepsilon=0$ in the sense that   all dynamical information on  $(\phi, c_1,c_2)$ would be lost when setting $\varepsilon=0$.  In \cite{Liu09},   the dependent variables are rescaled as  
\[u=\varepsilon p,\quad z_2c_2=-z_1c_1-Q_j-\varepsilon q.\]

 Replacing $(u,c_n)$ with $(p, q)$,  slow system~\eqref{slow} becomes  
 \begin{align*}
   \dot \phi=&  p,\quad 
  \varepsilon \dot p=  q, \quad   \dot c_1=-z_1c_1p -  \dfrac{J_1}{D_1 h(w)},\quad \dot J=0,\quad \dot w=1,\\
  \varepsilon \dot q=&\Big((z_1-z_2)z_1c_1 -z_2Q_j-\varepsilon z_2 q\Big)p+\frac{1}{h(w)}\Big(\dfrac{z_1J_1}{D_1}+\dfrac{z_2J_2}{D_2}\Big).
 \end{align*}
The limiting slow system  is
\begin{align}\label{limnewslow}\begin{split} 
  \dot \phi=&  p,\quad
  q=0,\quad  \dot c_1=-z_1c_1p -  \dfrac{J_1}{D_1h(w)},\quad \dot J=0,\quad \dot w=1,\\
 0=  & \Big((z_1-z_2)z_1c_1 -z_2Q_j\Big)p+\frac{1}{h(w)}\Big(\dfrac{z_1J_1}{D_1}+\dfrac{z_2J_2}{D_2}\Big).    
 \end{split}
 \end{align} 
Therefore, on the new slow manifold 
\[{\mathcal S}_j = \left\{p=-\frac{{z_1J_1}/{D_1}+{z_2J_2}/{D_2}}{h(w)\Big((z_1-z_2)z_1c_1 -z_2Q_j\Big)},\; q=0\right\},\]
 system~\eqref{limnewslow}  reads  
\begin{align}\label{redlimslow0}\begin{split}
  \dot \phi=& -\frac{{z_1J_1}/{D_1}+{z_2J_2}/{D_2}}{h(w)\Big((z_1-z_2)z_1c_1 -z_2Q_j\Big)},\\
    \dot c_1=&\frac{{z_1J_1}/{D_1}+{z_2J_2}/{D_2}}{h(w)\Big((z_1-z_2)z_1c_1 -z_2Q_j\Big)}z_1c_1
    -\dfrac{J_1}{D_1h(w)},\\
\dot J=&0,\quad \dot w=1.
\end{split}
 \end{align}
 On ${\mathcal S}_j$ where $q=z_1c_1+z_2c_2+Q_j=0$, it follows   that
\[(z_1-z_2)z_1c_1 -z_2Q_j
  = z_1^2c_1+z_2^2c_2.\] 

\begin{rem}\label{slowphi}   Note that, with {\em equal diffusion constant condition $D_1=D_2$},  the zero current $ I= z_1J_1+z_2J_2=0$ reduces system (\ref{redlimslow0}) to
\begin{align*}
  \dot \phi=& 0,\quad
    \dot c_1=     -  \dfrac{J_1}{D_1h(w)},\quad
\dot J=0,\quad \dot w=1.
 \end{align*}
The system can be solved explicitly and the solution is simple enough which is the very reason for the authors in \cite{ELX15} to obtain their rather specific results for general $n$. This is NOT the case if $D_k$'s are not the same. {In order to get reasonably explicit solution that can lead to advances of understanding of the physical problem, one has serious trouble to treat even the case with $n=2$.} In fact, we can only handle the case where $n=2$ and $z_1=-z_2$ at this moment.    \qed
\end{rem}

We now get back to system (\ref{redlimslow0}) and add a further assumption that $z_1=-z_2$. For zero current
  $I=z_1J_1 + z_2J_2 =0$ (so $J_1=J_2$), we have
\begin{equation}\label{Dz}
 \dfrac{z_1J_1}{D_1} + \dfrac{z_2J_2}{D_2} = \dfrac{D_2 - D_1}{D_1D_2}z_1J_1.
\end{equation}


\noindent Applying zero current condition \eqref{Dz}, the limiting slow system \eqref{redlimslow0} becomes,
\begin{align}\label{Dslow1}
\begin{split}
  \dot \phi=&  -\frac{(D_2-D_1)J_1}{D_1 D_2 h(\omega) (2z_1c_1 +Q_j)},\quad
    \dot c_1  =-\frac{(D_2 + D_1) z_1c_1+D_2Q_j }{D_1D_2 h(\omega) (2z_1c_1 +Q_j)} J_1,\quad
\dot J_1=0,\quad \dot w=1.
\end{split}
 \end{align}
 
 
 \noindent
 \underline{\bf Slow system (\ref{Dslow1}) on $(x_0,x_1)$  with $Q_1=0$:}
 \begin{align}\label{Dslow1x1}
  \dot \phi=&  -\frac{(D_2 - D_1)J_1}{2D_1 D_2 h(\omega)z_1c_1},\quad
    \dot c_1  =-\frac{D_1 +D_2 }{2D_1 D_2 h(\omega) } J_1,\quad
\dot J_1=0,\quad \dot w=1.
 \end{align}
The solution of \eqref{Dslow1x1} with the initial condition $(V, l_1, J_1, 0 )$ is,
\begin{align*}
&w (x) = x, \quad c_1(x) = l_1 - \frac{D_1 + D_2  }{2D_1 D_2} J_1 H(x),\quad \phi(x) = V - \frac{D_1 - D_2}{z_1(D_1 + D_2)} \ln \dfrac{c_1(x)}{l_1},
\end{align*}
where  $H(x) =\displaystyle \int_0^x \frac{1}{h(s)}ds$.
Evaluating the solution at $w = x =x_1$ we get the following lemma.
\begin{lem}\label{lem-slow1}
 Over $(0,x_1)$ with $z_1c_1(x) + z_2c_2(x)=-Q_1=0$ the slow dynamics system gives,
\begin{align*}
c_1^{[1,-]} = l_1 - \frac{D_1+D_2 }{ 2D_1 D_2} J_1 H(x_1),\quad
\phi^{[1,-]} = V - \frac{D_1 - D_2}{z_1(D_1+D_2)} \ln \dfrac{c_1^{[1,-]}}{l_1}.
\end{align*}
 \end{lem}
 
 \noindent
 \underline{\bf Slow system (\ref{Dslow1}) on $(x_1,x_2)$  with $Q_2\neq 0$:}
  Note that $h(w) >0$. Also, $c_k$'s are the concentrations of ion species. Therefore, we will be interested  in solutions with $c_k>0$ for $k=1,2$, and hence $ (z_1-z_2)z_1c_1 - z_2Q_2 = z_1^2c_1 + z_2^2c_2>0$. \\
 Hence, if we multiply $h(w) ((z_1-z_2)z_1c_1 - z_2Q_2) >0$ on the right hand side of system \eqref{Dslow1}, the   phase portrait remains the same and we have,
 \begin{align}\label{DSlow-y}
\begin{split}
  \dfrac{d}{dy} \phi=& \dfrac{D_1-D_2}{D_1D_2}z_1J_1,\quad
    \dfrac{d}{dy} c_1  =-\dfrac{(D_1+D_2)z_1^2J_1}{D_1 D_2}c_1+\dfrac{  z_1Q_2}{D_1}J_1,\\
\dfrac{d}{dy} J_1=&0,\quad  \dfrac{d}{dy} w=h(w) \big(2z_1^2c_1 + z_1Q_2\big) .
\end{split}
 \end{align}
The solution of \eqref{DSlow-y} with the initial condition $(\phi^{[1,+]}, c_1^{[1,+]}, J_1, x_1 )$ is,
 \begin{equation}\label{y-slow}
\begin{aligned}
\phi(y)=& \phi^{[1,+]}+\dfrac{D_1-D_2}{D_1D_2}z_1J_1y,\\
c_{1}(y)=& e^{-\frac{D_1+D_2}{D_1 D_2} z_1^2J_1y}c_1^{[1,+]} +\dfrac{D_2Q_2}{(D_1+{D_2}) z_1}\Big( e^{-\frac{D_1+D_2}{D_1 D_2} z_1^2J_1y}-1 \Big),\\
\int_{x_1}^{w} \dfrac{1}{h(s)}ds =&-\dfrac{2D_1 D_2z_1 c_1^{[1,+]}}{(D_1 +D_2) J_1}\Big(e^{-\frac{D_1+D_2}{D_1 D_2} z_1^2J_1y}-1 \Big)\\
&-\dfrac{2D_2 z_1Q_2}{D_1 + D_2 }\Big(D_1D_2\dfrac{ e^{-\frac{D_1+D_2}{D_1 D_2} z_1^2J_1y}-1}{(D_1+ D_2) z_1^2J_1}+y \Big)+z_1yQ_2.
\end{aligned}
\end{equation}
Assume $w (y^*)= x_2 $ for some $y^* >0$, then $\phi (y^*) = \phi^{[2,-]}$ and $c_1(y^*) =c_1^{[2,-]}$. Then, from \eqref{y-slow} one has the following result.
%
%

 \begin{lem}\label{lem-slow2}
 Over $(x_1,x_2)$ with $z_1c_1(x) + z_2c_2(x)+Q_2=0$ the slow dynamics system gives,
\begin{align*} 
&\phi^{[2,-]}= \phi^{[1,+]} + \dfrac{D_1 - D_2}{D_1 D_2} z_1J_1y^*,\\
& c_{1}^{[2,-]}= e^{-\frac{D_1+D_2 }{D_1 D_2} z_1^2J_1y^*}c_1^{[1,+]}+\dfrac{D_2Q_2}{(D_1 + D_2) z_1}\big( e^{-\frac{D_1+ D_2}{D_1 D_2} z_1^2J_1y^*}-1 \big),\\
&J_1= -D_1D_2\dfrac{2\big(c_{1}^{[2,-]} - c_{1}^{[1,+]} \big)-(\phi^{[2,-]} - \phi^{[1,+]})Q_2}{(D_1 + D_2)\big(H(x_2) - H(x_1) \big) } .
\end{align*}
 \end{lem}


 \noindent
 \underline{\bf Slow system (\ref{Dslow1}) on $(x_2,x_3)$  with $Q_3= 0$:}
 The slow dynamics system is \eqref{Dslow1x1} and the solution with the initial condition $(\phi^{[2,+]} , c_1^{[2,+]}, J_1, x_2 )$ is,
 \begin{align*}
c_1(x) =& c_1^{[2,+]} - \frac{D_1 + D_2 }{2D_1 D_2 } J_1 \big(H(x) - H(x_2) \big),\quad
 \phi(x)  = \phi^{[2,+]} - \frac{D_1 -D_2}{z_1(D_1+ D_2)} \ln \dfrac{c_1(x)}{c_1^{[2,+]}}.
\end{align*}
 
Evaluating the solution at $w = x =1$ we get the following lemma.
\begin{lem}\label{lem-slow3}
 Over $(x_2,1)$ with $z_1c_1(x) + z_2c_2(x)=0$ the slow dynamics system gives,
\begin{equation*}
 c_1^{[2,+]} = r_1 + \frac{D_1 +D_2  }{2D_1 D_2} J_1 \big(H(1) - H(x_2) \big),\quad
\phi^{[2,+]} =  \frac{D_1 -D_2}{z_1(D_1 + D_2)} \ln \dfrac{r_1}{c_1^{[2,+]}}.
 \end{equation*}
 \end{lem}


\subsection{Matching for Zero-current and Singular Orbits on $[0,1]$.}
The final step for the construction of a connecting orbit over the whole interval $[0,1]$ is to match the three singular orbits from previous section at the points $x=x_1$ and $x=x_2$. 
The matching conditions are $u_-^{[1]}=u_+^{[1]},~u_-^{[2]}=u_+^{[2]}$, and that $J_1$ has to be the same on all three subintervals. 
Recall that we only consider   the case where $n=2$ (two ion species) with $z_1 = -z_2$.  
It follows from Lemmas \ref{fast1}  and \ref{fast2}, and equations in Lemma \ref{lem-slow1} to Lemma \ref{lem-slow3},
\begin{equation}\label{Matching}
\begin{aligned}
&{ c_1^{[1]}e^{z_1(\phi^{[1]}-\phi^{[1,-]})}-c_2^{[1]}e^{-z_1(\phi^{[1]}-\phi^{[1,-]})}=0,}\\
& { c_1^{[2]}e^{z_1(\phi^{[2]}-\phi^{[2,+]})}- c_2^{[2]}e^{-z_1(\phi^{[2]}-\phi^{[2,+]})} =0,}\\
& { z_1c_1^{[1]}e^{z_1( \phi^{[1]}-\phi^{[1,+]})}- z_1c_2^{[1]}e^{-z_1( \phi^{[1]}-\phi^{[1,+]})}+ Q_2=0,}\\
&  { z_1c_1^{[2]}e^{z_1( \phi^{[2]}-\phi^{[2,-]})}- z_1c_2^{[2]}e^{-z_1( \phi^{[2]}-\phi^{[2,-]})}+ Q_2=0,}\\
&{2c_1^{[1,-]}}=c_1^{[1]}e^{z_1( \phi^{[1]}-\phi^{[1,+]})} + c_2^{[1]}e^{{-z_1}( \phi^{[1]}-\phi^{[1,+]})} +Q_2(\phi^{[1]}-\phi^{[1,+]}),\\
&{2c_1^{[2,+]}}=c_1^{[2]}e^{z_1( \phi^{[2]}-\phi^{[2,-]})} + c_2^{[2]}e^{{-z_1}( \phi^{[2]}-\phi^{[2,-]})} +Q_2(\phi^{[2]}-\phi^{[2,-]}),\\
&\dfrac{J_1}{D_1 D_2} = -\dfrac{2(c_1^{[1,-]}- l_1)}{(D_1+D_2 ) H(x_1)} = -\dfrac{2(r_1- c_1^{[2,+]})}{(D_1 + D_2  )(H(1) - H(x_2))},\\
&\hspace*{.41in} = -\dfrac{2z_1(c_1^{[2,-]}- c_1^{[1,+]}) - (\phi^{[2,-]} - \phi^{[1,+]})z_1Q_2}{(D_1 +D_2 )(H(x_2)-H(x_1))},\\
&\phi^{[2,-]}= \phi^{[1,+]} + \dfrac{D_1 - D_2}{D_1 D_2} z_1J_1y^*,\\
&c_{1}^{[2,-]}= e^{-\frac{D_1 + D_2}{D_1 D_2} z_1^2J_1y^*}c_1^{[1,+]}+\dfrac{D_2 Q_2}{(D_1 +D_2) z_1}\Big( e^{-\frac{D_1+ D_2}{D_1 D_2} z_1^2J_1y^*}-1 \Big),
\end{aligned}
\end{equation}
where,
\begin{equation}\label{Matching2}
\begin{aligned}
c_1^{[1,-]}=&c_1^{[1]}e^{z_1( \phi^{[1]}-\phi^{[1,-]})} = \sqrt{c_1^{[1]} c_2^{[1]}} ,\quad
c_2^{[1,-]}= c_2^{[1]}e^{z_2( \phi^{[1]}-\phi^{[1,-]})} = \sqrt{c_1^{[1]} c_2^{[1]}},\\
c_1^{[2,+]}=& c_1^{[2]}e^{z_1( \phi^{[2]}-\phi^{[2,+]})} =  \sqrt{c_1^{[2]} c_2^{[2]}},\quad
c_2^{[2,+]}=  c_2^{[2]}e^{z_2( \phi^{[2]}-\phi^{[2,+]})} ={ \sqrt{c_1^{[2]} c_2^{[2]}}},\\
c_1^{[1,+]}=&c_1^{[1]}e^{z_1( \phi^{[1]}-\phi^{[1,+]})}, \quad c_2^{[1,+]}=c_2^{[1]}e^{z_2( \phi^{[1]}-\phi^{[1,+]})},\\
c_1^{[2,-]}=&c_1^{[2]}e^{z_1( \phi^{[2]}-\phi^{[2,-]})}, \quad c_2^{[2,-]}=c_2^{[2]}e^{z_2( \phi^{[2]}-\phi^{[2,-]})},\\
\phi^{[1,-]} =& V - \frac{D_1 - D_2}{z_1(D_1  + D_2)} \ln \frac{c_1^{[1,-]}}{l_1},\quad \phi^{[2,+]} 
=  \frac{D_1 - D_2}{z_1(D_1 +D_2)} \ln \frac{r_1}{c_1^{[2,+]}}.
\end{aligned}
\end{equation}

\begin{rem} In \eqref{Matching}, the unknowns are: $\phi^{[1]},~\phi^{[2]},~ c_1^{[1]},~c_2^{[1]},~ c_1^{[2]},~c_2^{[2]},~J_1,~\phi^{[1,+]},~\phi^{[2,-]},~y^*$ and $Q_2$  that is, there are eleven unknowns that matches the total number of equations on \eqref{Matching}. \qed
\end{rem}

It follows from last two equations of \eqref{Matching2},
\begin{equation}\label{phi_al,br}
\begin{aligned}
&\phi^{[1]}- \phi^{[1,-]} = \dfrac{D_1 +D_2}{2D_2} \big(\phi^{[1]}- V \big) + \frac{D_1 - D_2}{2z_1D_2}\ln \dfrac{c_1^{[1]}}{l_1} ,\\
&\phi^{[2]}- \phi^{[2,+]} = \dfrac{D_1  +D_2 }{2D_2} \phi^{[2]} - \frac{D_1 - D_2}{2z_1D_2 }\ln \dfrac{r_1}{c_1^{[2]}} .
\end{aligned}
\end{equation}


%
%


\section{Reduced System for Zero-current with $z_1=-z_2>0$.}\label{sec-ReducedSys}
\setcounter{equation}{0}
The matching system \eqref{Matching} is nonlinear and challenging to analyze in general. In \cite{ELX15}, for equal diffusion constants $D_k$'s, the study of reversal potential and reversal permanent charges has been successfully carried out for a general $n$.   With general $D_k$'s the problem becomes overwhelmingly harder, at least, technically, even for the case that we will treat here where $n=2$ with $z_1=-z_2$.  


In \cite{EL07}, the authors introduced two intermediate variables that allow a significant reduction of the governing system of matching (\ref{Matching}) without zero current assumption. We will use the same intermediate variables for our reduction. Thus, we set
 \begin{align}\label{AandB}
 A=\sqrt{c_1^{[1]}c_2^{[1]}}\;\mbox{ and }\; B=\sqrt{c_1^{[2]}c_2^{[2]}}.
 \end{align}
 Note that $A$ and $B$ are the geometric mean of concentrations at $x=x_1$ and $x=x_2$ respectively.
 It will be shown in (\ref{Bphi})in the appendix  that  $B=B(A) = \dfrac{1-\beta}{\alpha}(l-A)+r$. We will thus treat $B$ as a function of $A$ instead of an independent variable from now on.
 We  denote
 \begin{align}\label{lrp}
 l_1=l_2=l, \quad r_1=r_2=r, \quad Q_2=2Q_0,\quad \alpha = \dfrac{H(x_1)}{H(1)}, \quad \beta = \dfrac{H(x_2)}{H(1)}, \quad \theta=\frac{D_2- D_1}{D_2 + D_1}.
 \end{align}
 Note that $l$ and $r$ are the concentrations of the ionic solutions at the left and right bathes, respectively. Recall $H(x)=\int_0^xh^{-1}(s)ds$ with $h(x)$ being the (dimensionless) area of the cross-section of the channel over $x$.  In the simplest case where $h(x)$ is a constant, $H(x)$ is the ratio between the length of the portion $[0,x]$ of the channel over the cross-section area of the channel.   So $0<\alpha<\beta<1$ are normalized factor associated to the potions $[0,x_1]$ and $[0,x_2]$, respectively (see, e.g. discussions at the end of Section 2 in \cite{ZEL}). The quantity $\theta\in (-1,1)$ is a measurement of unequal of $D_1$ and $D_2$, in particular, $\theta=0$ if and only if $D_1=D_2$.
 
 The vector  $(Q_0, V, \theta, \alpha, \beta,  l,r)$ contains major parameters which affect the behavior of the system through their nonlinear interactions.  In the sequel, we will always fix the parameters $\alpha$, $\beta$, $l$ and $r$, and focus on the roles of $(V, Q_0, \theta)$.  One can see that  the roles of $(\alpha, \beta, l,r)$ can be studied within our analysis framework.
    For ease of notation,  we also introduce
 \begin{align}\label{abp} 
  S_a:= \sqrt{Q_0^2+z_1^2A^2}, \quad S_b:= \sqrt{Q_0^2+z_1^2B^2}, 
 \quad N: =  \dfrac{\beta - \alpha}{\alpha}z_1(A-l)  + S_a -S_b.
\end{align}

The most critical ingredient for our analysis is the following result on a reduced system of the matching system \eqref{Matching}.
\begin{prop}\label{rSys} The matching system \eqref{Matching}  for zero current $I=0$ can be reduced to 
\begin{align}\label{G1G2Sys}
G_1(A, Q_0,\theta)=z_1V\;\mbox{ and }\; G_2(A,Q_0,\theta)=0,
\end{align}
  where
\begin{equation}\label{G}
\begin{aligned}
G_1(A, Q_0,\theta)=& \theta \Big(\ln\dfrac{S_a + \theta Q_0}{S_b + \theta Q_0} + \ln\dfrac{l}{r}\Big)  - (1+\theta)\ln \dfrac{A}{B} + \ln \dfrac{S_a -Q_0}{S_b -Q_0},\\
G_2(A, Q_0,\theta)=& \theta Q_0\ln\dfrac{S_a+\theta Q_0}{S_b+\theta Q_0}-N.
\end{aligned}
\end{equation}	
\end{prop}
\begin{proof}  We defer the proof to the appendix Section \ref{Appendix}.
\end{proof}

At this moment, we would like to make some comments on the above reduction. 

\begin{rem}\label{RMatch}    The   reduction of \eqref{Matching} to system (\ref{G1G2Sys})   is critical for the remaining analysis. We   comment   that there is no practical principle to lead the reduction and no criterion for a `good' final form of a reduction. In general, there could be infinitely many different forms of the reduction.     It turns out the above reduction works well.
	
{ For a uniform $h(x)$, the quantity $H(x)$ is the ratio of the length with the cross-section area of the potion of the channel over $[0, x]$. The quantity $H(x)$ has its origin in Ohm law for
	the resistance of a uniform resistor. It appears that the quantities $a$ and $b$ together with the value $Q_0$ are the chief characteristics of the shape and permanent charge of the channel formation.}
 
For the special case where $h=1$, $x_1=1/3$, $x_2=2/3$, $z_1=1=-z_2$, and $D_1=D_2$, a reduced system consists of $F(A)=0$ in (48) in  \cite{EL07} and $I=0$. One can get different equivalent forms and,   as expected, one equivalent reduced system can be put into exactly the same as the one stated in Proposition  \ref{rSys}. We also note that, for a given $Q_0$, one cannot solve for $A$ from either  $F(A)=0$ or $I=0$ uniquely. But, we will show that one can solve for $A$ from $G_2=0$ uniquely -- a critically important indication that the specific form of system (\ref{G1G2Sys}) is special. 
 \qed
\end{rem}

  We now prepare several properties of the functions $G_1 $ and $G_2$ to be used later on.
\begin{lem}\label{lem-parG1G2} One has 
\begin{itemize}
\item[(i)] $\partial_{A}G_1(A, Q_0,\theta)$ has the same sign as that of $Q_0$,
\item[(ii)] $\partial_{Q_0}G_1(A, Q_0,\theta)$ has  the same sign as that of $l-r$, 
\item[(iii)] $\partial_{\theta} G_1 (A, Q_0,\theta)$ has the same  sign as that of $l-r$,
\item[(iv)] $\partial_A G_2 (A, Q_0,\theta)<0$,
\item[(v)] if $\theta Q_0> 0$, then $\partial_{Q_0} G_2(A, Q_0,\theta)$ has the same sign as that of $(l-r)Q_0$,
\item[(vi)] $\partial_{\theta} G_2 (A, Q_0,\theta)$ has the same  sign as that of $(l-r)Q_0$.
\end{itemize}
\end{lem}
\begin{proof}
Partial derivatives of $G_1$ and $G_2$ with respect to $Q_0$ and $A$ are,
\begin{align}\label{parG1G2}\begin{split}
\partial_{A}G_1(A, Q_0,\theta) =& (1-\theta^2)Q_0\Big(\dfrac{1}{A(S_a+\theta Q_0)} +\frac{1-\beta}{\alpha} \dfrac{1}{B(S_b+\theta Q_0)}\Big),\\
\partial_{Q_0}G_1(A, Q_0,\theta) =&  \dfrac{(1-\theta^2)(S_a-S_b)}{(S_a+\theta Q_0)(S_b+\theta Q_0)} ,\\
\partial_{\theta} G_1 (A, Q_0,\theta) =&g(S_a)-g(S_b) +\ln \dfrac{l}{r} - \ln \dfrac{A}{B},\\
 \partial_A G_2 (A, Q_0,\theta)=& -\frac{1-\beta}{\alpha}\dfrac{z_1^2B}{S_b+\theta Q_0} - \dfrac{z_1^2A}{S_a+\theta Q_0} - \dfrac{\beta - \alpha}{\alpha}z_1,\\
\partial_{Q_0} G_2(A, Q_0,\theta) =&\theta \ln\dfrac{S_a+\theta Q_0}{S_b+\theta Q_0}+\dfrac{(1-\theta^2)Q_0(S_a-S_b)}{(S_a+\theta Q_0)(S_b+\theta Q_0)},\\
\partial_{\theta} G_2 (A, Q_0,\theta) =&Q_0(g(S_a)-g(S_b)), 
\end{split}
\end{align} 
where 
\begin{equation}\label{gg}
g(X) := \ln (X+\theta Q_0)   + \dfrac{\theta Q_0}{X+\theta Q_0}.
\end{equation}
All statements except those for  signs of $\partial_{\theta}G_k$'s follow directly from (\ref{parG1G2}).   For  signs of $\partial_{\theta}G_k$'s, note that
  \[g'(X)=\dfrac{X}{(X+\theta Q_0)^2}>0\;\mbox{ for }\;X>0.\]
  So $g(S_a)-g(S_b)$ has the same sign as that of $S_a-S_b$.
    It is obvious that  $S_a-S_b$ has the same sign as that of $A-B$ and  it will be shown in Theorem \ref{AonQ} that $l-r$ and $A-B$ have the same sign too. The statements on the signs of $\partial_{\theta} G_k$'s then follow.
  \end{proof}

 
 \subsection{The Solution $A=A(Q_0,\theta)$  of $G_2(A,Q_0,\theta)=0$.}

Recall  that { $A$ and $B$ are the geometric mean of concentrations at $x=x_1$ and $x=x_2$ respectively, and} $B= \dfrac{1-\beta}{\alpha}(l-A)+r$. One has $B=A$ if and only if
{$\displaystyle{A=A^*=\frac{(1-\beta)l+\alpha r}{1-\beta+\alpha}}$}.
It is clear that $l<A^*<r$ if $l<r$ and $l>A^*>r$ if $l>r$.

\begin{thm}\label{thmA(Q)} For any given $(Q_0,\theta)$,   $G_2(A,Q_0,\theta)=0$ has a unique solution $A=A(Q_0,\theta)$. 
\end{thm}
\begin{proof} For any $(Q_0,\theta)$, it follows from Lemma \ref{lem-parG1G2} that $\partial_AG_2(A,Q_0,\theta)<0$, and hence, $G_2(A,Q_0,\theta)$ is strictly decreasing in $A$. Let $A_M=l+\alpha r/{(1-\beta)}$ be the largest value for $A$ (when $B=0$) and let $B_M=(1-\beta)l/{\alpha}+r$ be the largest value for $B$ (when $A=0$).

Set   $x=\sqrt{Q_0^2+z_1^2B_M^2}>|Q_0|$ and $y=\sqrt{Q_0^2+z_1^2A_M^2}>|Q_0|$. Then,
\begin{align*}
G_2(0^+,Q_0,\theta)=&f_1(x):= \theta Q_0\ln\frac{ |Q_0|+\theta Q_0}{ x+\theta Q_0}+  \dfrac{\beta - \alpha}{\alpha}z_1l  -|Q_0| +x,\\
G_2(A_M^-,Q_0,\theta)=&f_2(y):= \theta Q_0\ln\frac{ y+\theta Q_0}{|Q_0|+\theta Q_0}-  \dfrac{\beta - \alpha}{\alpha}z_1(A_M-l) -y +|Q_0|. 
\end{align*}
 It is easy to check that $f_1'(t)>0>f_2'(t)$ for $t>0$, and hence,   
 \[f_1(x)>f_1(|Q_0|)=\dfrac{\beta - \alpha}{\alpha}z_1l>0\; \mbox{ and }\;
 f_2(y)<f_2(|Q_0|)=-\dfrac{\beta - \alpha}{\alpha}z_1(A_M-l)<0.\]
Thus, for any $(Q_0,\theta)$  there is a unique $A=A(Q_0,\theta)$ such that $G_2(A(Q_0,\theta), Q_0,\theta)=0$. 
\end{proof}

In the following,   we also  denote $B(A(Q_0,\theta))$ by  $B(Q_0,\theta)$.
\begin{thm}\label{AonQ}  The solution $A=A(Q_0,\theta)$ of $G_2(A,Q_0,\theta)=0$ satisfies  
\begin{itemize}
\item[(a)] $A(0,\theta) = (1-\alpha)l + \alpha r$  and $\lim_{Q_0 \to \pm \infty} A(Q_0,\theta) = l$,
\item[(b)] if $l>r$, then $ l>A(Q_0,\theta)> A^*>B(Q_0,\theta)>r$,
\item[(c)] if $l<r$, then  $l<A(Q_0,\theta)< A^*<B(Q_0,\theta)<r$,
\item[(d)] if   $\theta Q_0\ge 0$, then   $\partial_{Q_0}A(Q_0,\theta)$ has the same sign as that of $(l-r)Q_0$.
\end{itemize}
\end{thm}
\begin{proof}  (a). The value $A(0,\theta)$ can be deduced from
 \[G_2(A,0,\theta)=-\dfrac{\beta - \alpha}{\alpha}z_1(A-l)-  z_1(A-B)=0\;\mbox{ and }\; B = \dfrac{1-\beta}{\alpha}(l-A)+r.\]

For the claim about the limits, one has, from $G_2(A(Q_0,\theta),Q_0,\theta)=0$,  
\begin{align*}
 \lim_{Q_0\to \pm \infty} \theta Q_0\ln\dfrac{S_a + \theta Q_0}{S_b + \theta Q_0}=&\lim_{Q_0\to \pm \infty}\Big(\dfrac{\beta - \alpha}{\alpha}z_1(A-l)  +S_a -S_b\Big)
= \dfrac{\beta - \alpha}{\alpha}z_1 \lim_{Q_0\to \pm \infty} (A-l).
\end{align*}
On the other hand, apply L'Hospital rule to get 
$$
\begin{aligned}
\lim_{Q_0\to \pm \infty}Q_0 \ln\dfrac{S_a+\theta Q_0}{S_b+ \theta Q_0} =& -\lim_{Q_0\to \pm \infty}\dfrac{(\frac{Q_0}{S_a}+\theta)(S_b+\theta Q_0)-(\frac{Q_0}{S_b}+\theta)(S_a+\theta Q_0)}{(S_a+\theta Q_0)(S_b+\theta Q_0)} Q_0^2=0.
\end{aligned}
$$
Thus, $\lim_{Q_0 \to \pm \infty} A(Q_0,\theta) = l$.

(b).  Recall that, for $A=A^*=\frac{(1-\beta)l+\alpha r}{1-\beta+\alpha}$, $B=A^*$. Thus, 
\[G_2(A^*,Q_0,\theta)=\frac{\beta-\alpha}{\alpha}z_1(l-A^*)=\frac{\beta-\alpha}{1-\beta+\alpha}z_1(l-r).\] 
Note that,   for some $S_*$ between $S_a$ and $S_b$,
\[\ln\dfrac{S_a+\theta Q_0}{S_b+\theta Q_0}=\ln(S_a+\theta Q_0)-\ln(S_b+ \theta Q_0)=\frac{S_a-S_b}{S_*+\theta Q_0}.\]
Thus, for some $S_*$ between $S_a$ and $S_b$,
\[G_2(l,Q_0,\theta)=\theta Q_0\ln\dfrac{S_a+ \theta Q_0}{S_b+\theta Q_0}- (S_a-S_b)=-\frac{(S_a-S_b)S_*}{S_*+\theta Q_0}.\]

If $l>r$, then $G_2(A^*,Q_0,\theta)>0$, which yields $A^*<A(Q_0,\theta)$ since $G_2$ is decreasing in $A$. The latter implies $B(Q_0,\theta)<A^*<A(Q_0,\theta)$, and hence, $S_a>S_b$, which then implies $G_2(l,Q_0,\theta)<0$. 
Due to again that $G_2$ is deceasing in $A$,  $r<A^*<A(Q_0,\theta)<l$ if $l>r$  (independent of $Q_0$).

(c). Similarly, if $l<r$, then $G_2(A^*,Q_0,\theta)<0<G_2(l,Q_0,\theta)$, and hence, $A^*>A(Q_0,\theta)>l$.

(d).   It follows from (\ref{parG1G2}) that, if $\theta Q_0>0$ or $\theta=0$, then $\partial_{Q_0} G_2$ has the same sign as that of $(S_a-S_b)Q_0$. The latter has the same sign as that of $(l-r)Q_0$. The statement then follows from
$\partial_A G_2 <0$ and   $\partial_{Q_0}A=-\partial_{Q_0} G_2/{\partial_A G_2}$.
 \end{proof}


\begin{rem}\label{monoG2}  Note that, with zero current condition $I=0$, we have that $A(Q_0,\theta)$ always lies between $l$ and $r$ for any $Q_0$. This is not true without zero current condition (see \cite{ZEL}).

We believe that, if $l\neq r$, then   $A(Q_0,\theta)$, or equivalently, $G_2(A(Q_0,\theta),Q_0,\theta)$ has a unique critical point in $Q_0$. It is true if $D_1=D_2$ (so $\theta=0$) but we could not establish it in general. However, numerical simulations support our belief that $A(Q_0,\theta)$ has a unique critical point in $Q_0$. 
 \qed
\end{rem}


%

\begin{thm}\label{Aonp} The quantity $\partial_{\theta} A(Q_0,\theta)$ has the same sign as that of $(l-r)Q_0$.
\end{thm}
\begin{proof}
 It follows from $G_2 =0$ in (\ref{G}) that,
$\displaystyle{\partial_{\theta} A = - {\partial_{\theta} G_2}/{\partial_A G_2}}$.
  The statement then follows from (iv) and (vi) in Lemma \ref{lem-parG1G2}. 
  \end{proof}

\subsection{ Zero-current Flux $J$.} For the case of zero current with $z_1=-z_2$, one has $J_1=J_2$. Denote the equal fluxes by $J$ that we call it {\em zero current flux}.  
Once a solution $(A,V)$ of $G_1=z_1V$ and $G_2=0$   is obtained, it follows from (\ref{J}) that   $ J$ is given by
 \begin{align}\label{Js}
 \begin{split}
J(Q_0,D_1, D_2) =& -\dfrac{2D_1 D_2(A- l)}{(D_1 + D_2)\alpha H(1)} = -\dfrac{2D_1 D_2(r- B)}{(D_1 + D_2) (1-\beta)H(1)}.
  \end{split}
\end{align}
  
  Note that the function $A=A(Q_0,\theta)$ depends on $D_1$ and $D_2$ through $\theta=(D_2-D_1)/{(D_2+D_1)}$ so $A$ is homogeneous of degree zero in $(D_1,D_2)$ but $J(Q_0,D_1, D_2)$ is not. 

The following   result is a direct consequence of Theorems \ref{AonQ} and \ref{Aonp}.
\begin{cor}\label{JonQp} The { zero-current} flux $J=J(Q_0,D_1,D_2)$ satisfies
\begin{itemize}
\item[(a)] if   $\theta Q_0\ge 0$, then   $\partial_{Q_0}J$ and $(l-r)Q_0$ have opposite signs,
\item[(b)] if $Q_0>0$, then   $\partial_{D_1}J$ and $l-r$ have the same sign,
\item[(c)] if $Q_0<0$, then   $\partial_{D_2}J$ and $l-r$ have the same sign.
 \end{itemize}
\end{cor}
\begin{proof} Direct calculations from (\ref{Js}) give 
\begin{align*}
\partial_{Q_0}J=& -\dfrac{2D_1 D_2 }{(D_1 + D_2)\alpha H(1)}\partial_{Q_0}A,\quad
\partial_{D_1}J=\dfrac{(1+\theta)^2}{2\alpha H(1)}\Big( l-A(Q_0,\theta)+(1-\theta) \partial_{\theta}A   \Big),\\
\partial_{D_2}J=&\dfrac{(1-\theta)^2}{2\alpha H(1)}\Big(l-A(Q_0,\theta)-(1+\theta) \partial_{\theta}A  \Big).
\end{align*}
The statement follows from the above formulas and Theorems \ref{AonQ} and \ref{Aonp}.
\end{proof}

	A non-intuitive outcome of the equations in \eqref{Js}, that can also be seen in Figure \ref{Fig-ZeroJvsQ1}, is that \medskip
	
\hspace*{1.4in}{\em The zero-current $J$ has the same sign as that of $l-r$.}
\medskip
	
	Moreover, note that if $D_1 = D_2$, then the zero-current flux $J$ is an even function in $Q_0$, and it is monotonic for $Q_0 > 0$. But, if $D_1 \neq D_2$, then the zero-current flux $J$ is not an even function in $Q_0$ and monotonicity of the zero-current flux $J$ in $Q_0$ is complicated. See the sections 2.1.1 and 2.1.2 in \cite{MEL20} for more results on the zero-current flux.

\begin{figure}[h]\label{Fig-ZeroJvsQ1}
	\centerline{\epsfxsize=3.0in \epsfbox{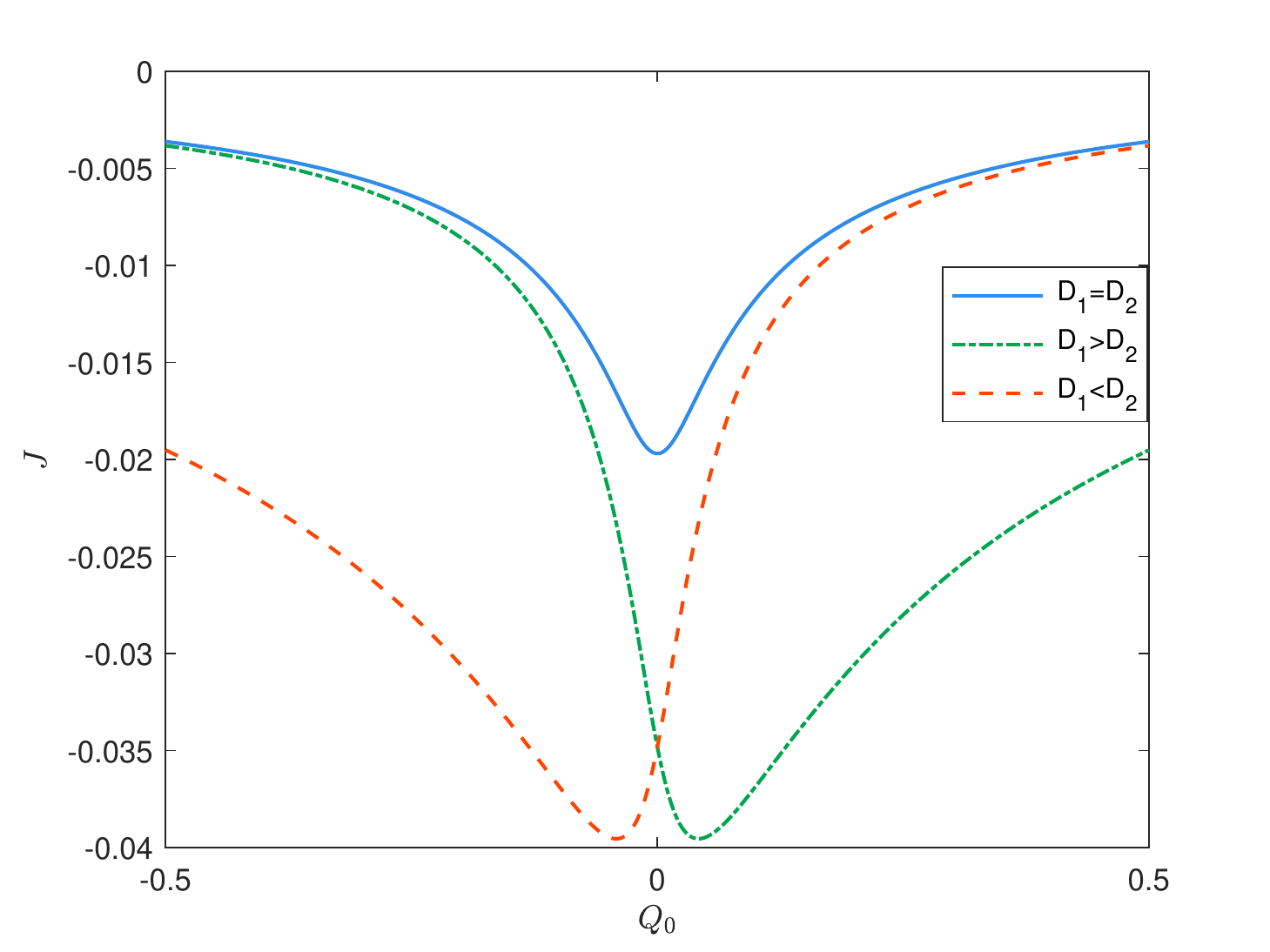} }
	\caption{\em  The zero-current flux  $J=J(Q_0)$ for various values of  $(D_1,D_2)$ when $l <r$,  which is negative; it has the same sign as that of  $l-r$, no matter what the values of diffusion constants are. When $D_1= D_2$, the zero-current flux is symmetric respect to $Q_0$ and  loses its symmetry when $D_1\neq D_2$.}
	\label{Fig-ZeroJvsQ1}
\end{figure}

\begin{rem}\label{signJ}  
We do not know the signs of $\partial_{Q_0}J$, $\partial_{D_1}J$, and $\partial_{D_2}J$ under conditions other than those in the statement of Corollary \ref{JonQp} in general.  
 \qed
\end{rem}



 \section{Reversal Potential $V_{rev}=V_{rev}(Q_0,\theta)$.}\label{revP}
\setcounter{equation}{0}
We are searching for  the value $V=V_{rev}$ of the transmembrane potential 
$V = \phi(0) - \phi(1)$ that produces zero current $I$.
 For the case we considered, we will show the existence and uniqueness of reversal potentials. 
\begin{thm}\label{thm-revpot} Consider ionic flow of two  ($n=2$) ion species  with $z_1=-z_2$. For any  given $(Q_0,\theta)$,
  there exists a unique reversal potential $V_{rev}=V_{rev}(Q_0,\theta)$.
\end{thm}
\begin{proof}
It follows from Theorem \ref{thmA(Q)} that for any given $(Q_0,\theta)$, there exists a unique $A=A(Q_0,\theta)$ such that $G_2(A(Q_0,\theta), Q_0,\theta)=0$. A reversal potential 
$V=V_{rev}$ is then  determined from $z_1V_{rev}=G_1(A(Q_0,\theta), Q_0,\theta)$ and is given by 
\begin{align}\label{VonQf}
V_{rev} (Q_0,\theta)= \dfrac{\theta}{z_1}   \big(\ln\dfrac{S_a + \theta Q_0}{S_b + \theta Q_0} + \ln\dfrac{l}{r}\big)  - \dfrac{1+\theta}{z_1} \ln \dfrac{A(Q_0,\theta)}{B(Q_0,\theta)} +\frac{1}{z_1} \ln \dfrac{S_a -Q_0}{S_b -Q_0},
\end{align}
where $S_a$ and $S_b$ are given in terms of $Q_0$ and $A (Q_0,\theta)$ as in (\ref{abp}).
 \end{proof}
 Based on formula (\ref{VonQf}) and definitions of $S_a$ and $S_b$ in (\ref{abp}), the following statement can be obtained directly. 
 \begin{cor}\label{VQ} For  $Q_0=0$, the reversal potential is  $V_{rev} (0,\theta)= \dfrac{\theta}{z_1}  \ln\dfrac{l}{r}$.
\end{cor}
In this case of zero permanent charge, for equal diffusion coefficients $D_1=D_2$ (so that $\theta=0$), the reversal potential is zero independent of the concentrations $l$ and $r$ at both ends. But, for unequal diffusion coefficients, the reversal potential is generally nonzero, which indicates that an electric field is needed to balance the diffusion created by unequal ionic mobilities.

\subsection{Dependence of the Reversal Potential $V_{rev}$  on $Q_0$.}
We will consider how the reversal potential $V_{rev}= V_{rev}(Q_0,\theta)$ depends on $Q_0$.  Recall that we denote $J_1=J_2$ by $J$.  
 \begin{thm}\label{VonQ} For the   reversal potential  $V_{rev}=V_{rev}(Q_0,\theta)$, one has 
 
 (i) if $l>r$, then $J>0$, and hence, $- \dfrac{1}{z_1}\ln \dfrac{l}{r}<  V_{rev}(Q_0,\theta)< \dfrac{1}{z_1}\ln \dfrac{l}{r}$;

 (ii) if $l<r$, then $J<0$, and hence, $\dfrac{1}{z_1}\ln \dfrac{l}{r}<  V_{rev}(Q_0,\theta)< -\dfrac{1}{z_1}\ln \dfrac{l}{r}$;
 
(iii) $ V_{rev}(0,\theta)=\dfrac{\theta}{z_1}\ln \dfrac{l}{r}$ and
$ \lim_{Q_0 \to \pm\infty}  V_{rev}(Q_0,\theta)= \pm \dfrac{1}{z_1}\ln \dfrac{l}{r}$.
\end{thm}
\begin{proof}  (i) It follows from  part (b) in Theorem \ref{AonQ} and the formula for $J$ in (\ref{Js}) that, if $l>r$, then $J>0$.
The range for $V_{rev}$ is a consequence of that fact that $J_k$ has the same sign as that of $z_kV+\ln l/r$.   
(ii) can be established similarly.

(iii) The value of $V_{rev}(0,\theta)$ is recast from  Corollary \ref{VQ} that follows from (\ref{VonQf}) directly. To show the limits, we recall from Theorem \ref{thmA(Q)} that
  $\lim_{Q_0\to \pm \infty}A(Q_0,\theta)=l$ (and hence, $\lim_{Q_0\to \pm \infty}B(Q_0,\theta)=r$). Note also that $-1<\theta<1$. Therefore, 
\[\lim_{Q_0\to  +\infty}\ln \dfrac{A(Q_0,\theta)}{B(Q_0,\theta)}=  \ln \dfrac{l}{r}, \quad \lim_{Q_0\to \pm\infty}\ln\dfrac{S_a + \theta Q_0}{S_b + \theta Q_0}=0, \]
and
\[\lim_{Q_0\to +\infty} \ln \dfrac{S_a -Q_0}{S_b -Q_0} = 2 \ln \dfrac{l}{r}, \quad \lim_{Q_0\to -\infty} \ln\dfrac{S_a -Q_0}{S_b -Q_0} = 0.\]
 Using (\ref{VonQf}), one then has
\begin{align*}
\lim_{Q_0\to +\infty} z_1  V_{rev}(Q_0,\theta) =& \theta \ln \dfrac{l}{r} - (1+\theta) \ln \dfrac{l}{r} + 2 \ln \dfrac{l}{r} = \ln \dfrac{l}{r},\\
\lim_{Q_0\to -\infty} z_1  V_{rev}(Q_0,\theta) =& \theta \ln \dfrac{l}{r} - (1+\theta) \ln \dfrac{l}{r}  =- \ln \dfrac{l}{r}.
\end{align*}
 The proof is completed. 
\end{proof}

  
The next result is  a direct consequence  of Theorem \ref{VonQ}, whose proof will be omitted. 
 \begin{cor}\label{V_rev-Sgn} One has, 
  \begin{itemize}
  \item[(i)] if $D_1<D_2$, then, for some $Q_0<0$,  $V_{rev}(Q_0,\theta)=0$;
  
 \item[(ii)] if $D_1>D_2$, then, for some $Q_0>0$,  $V_{rev}(Q_0,\theta)=0$.
  \end{itemize}
\end{cor}

We now provide remarks on the  physical basis for results in  Theorem \ref{VonQ} and Corollary \ref{V_rev-Sgn}.
   
 \begin{rem}\label{limRV}   The statements (i) and (ii) in  Theorem \ref{VonQ}  can be obtained in a direct way as follows. Note that, in general,  $J_k$ has the same sign as that of $z_kV+\ln l/r$. Thus, if $l>r$, then for $V\le V_1=- \dfrac{1}{z_1}\ln \dfrac{l}{r}$, one has $J_1\le 0$ since $z_1V+\ln l/r\le z_1V_1+\ln l/r= 0$, and $J_2>0$ since  $z_2V+\ln l/r=-z_1V+\ln l/r\ge -z_1V_1+\ln l/r=2\ln l/r>0$. Therefore, if $l>r$, then $V_{rev}>- \dfrac{1}{z_1}\ln \dfrac{l}{r}$ and,  similarly, $V_{rev}(Q_0,\theta)< \dfrac{1}{z_1}\ln \dfrac{l}{r}$.   
 
  In \cite{ZEL}, it shows that, as $Q_0\to +\infty$, $J_1(Q_0)\to 0$. Thus, $J_2(Q_0)\to 0$ from $I(Q_0)=0$. 
Since $J_2$ is proportional to $-z_1V_{rev}(Q_0,\theta)+\ln l/r$ with a positive proportional constant in general (see \cite{ZEL}), it follows that $-z_1V_{rev}(Q_0,\theta)+\ln l/r\to 0$ as $Q_0\to +\infty$, which is exactly what claimed in Theorem \ref{VonQ} for this limit. The other claim   follows from the same argument.
 
{ Statement (iii)} in Theorem \ref{VonQ} says that, if $D_1<D_2$ and $l>r$,  then  $V_{rev}(0,\theta)>0$. This makes sense since, for $V=0$ and $l>r$,  $J_2>0$ and $J_1>0$, and, with $D_1<D_2$, $J_1<J_2$. To help $J_1$ more than $J_2$ to get $J_1=J_2$, one needs to increase $V$ and this is why, in this case, $V_{rev}(0,\theta)>0$. The latter often implies that, if $V=0$, then $I(V=0)<0$, or equivalently, $J_2>J_1$. Thus, intuitively, in order for the  zero potential to be a reversal potential, a permanent charge helping $J_1$ more than $J_2$ is needed; that is,  the permanent charge should be negative, which agrees with statement (i) in Corollary \ref{V_rev-Sgn}.  Other statements in Corollary \ref{V_rev-Sgn} can be explained  similarly.  \qed
\end{rem}

 Concerning the monotonicity of $V_{rev}=V_{rev}(Q_0,\theta)$ in $Q_0$,  we have
 \begin{thm}\label{monoVinQ} For any given $\theta \in (-1,1)$, one has
 \medskip
 
%
%

 if $\theta Q_0 \geq 0$, then $V_{rev}(Q_0,\theta )$ is   increasing in $Q_0$ for $l>r$ and decreasing in $Q_0$ for $l<r$.

 \end{thm}
 \begin{proof} It follows from $z_1V_{rev}=G_1(A(Q_0,\theta),Q_0,\theta)$   and $\partial_{Q_0}A=-{\partial_{Q_0}G_2}/{\partial_{A}G_2}$ that
   \begin{align}\label{dVQ}
 \partial_{Q_0}V_{rev}  =& \dfrac{1}{z_1 \partial_{A}G_2} \Big( \partial_{Q_0}G_1 \partial_{A}G_2 -\partial_AG_1 \partial_{Q_0}G_2 \Big).
\end{align}
The statements then follow from Lemma \ref{lem-parG1G2}.
 \end{proof}

We conjecture that $V_{rev}(Q_0,\theta)$ is always monotonic in $Q_0$ but could not prove it.   Numerical simulations in Figure \ref{fig-VrevQ1} support  our conjecture.

\begin{figure}[ht]\label{fig-VrevQ1}
\centerline{\epsfxsize=3.0in \epsfbox{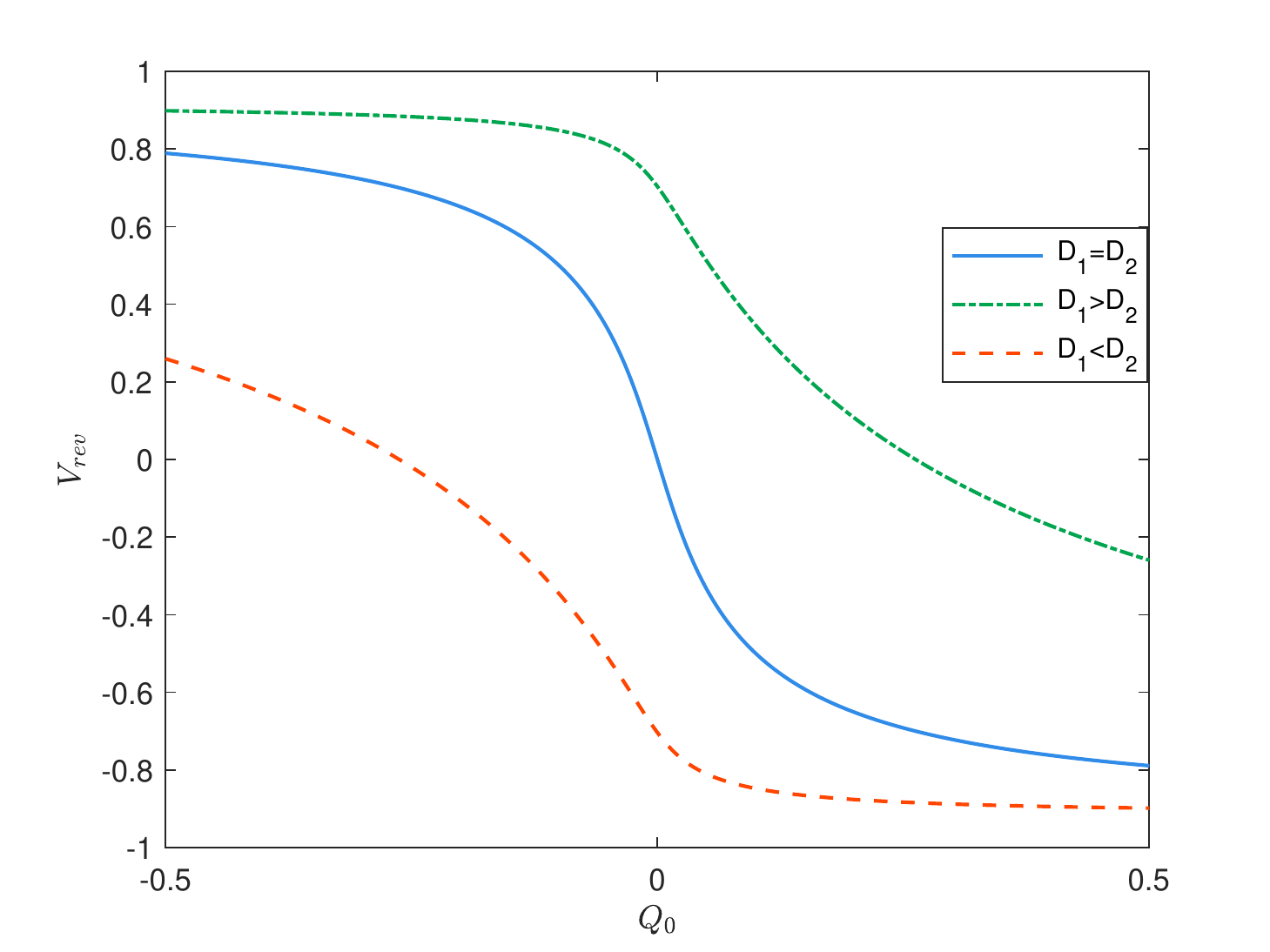} }
\caption{\em The function $V=V_{rev}(Q_0,D_1,D_2)$ for various values of $(D_1,D_2)$: it is an odd function in $Q_0$ if $D_1=D_2$ and the symmetry breaks if $D_1\neq D_2$.}
\label{fig-VrevQ1}
\end{figure}

For $|Q_0|$ small, we have  
\begin{thm}\label{RevInQ}  Near $Q_0=0$, the reversal potential $V_{rev}(Q_0,\theta)$ is approximated by  
\begin{align}\label{partQ-Vrev}
V_{rev}(Q_0,\theta) = \dfrac{\theta}{z_1}   \ln \dfrac{l}{r} + \dfrac{1-\theta^2}{z_1^2}\dfrac{(\beta - \alpha)(l-r)}{\big((1-\alpha) l + \alpha r \big)\big((1-\beta) l + \beta r \big)}Q_0 + O(Q_0^2).
\end{align}
\end{thm}
\begin{proof}  
It follows from (\ref{parG1G2}) that 
\[ \partial_{A}G_1(A,0,\theta)=0\;\mbox{ and }\; \partial_{Q_0}G_1(A(0,\theta),0,\theta)=\dfrac{1-\theta^2}{z_1}\dfrac{A(0,\theta)- B(0,\theta)}{A(0,\theta) B(0,\theta)}.\]
Recall that $A(0,\theta)=(1-\alpha) l + \alpha r$ and $B(0,\theta) = (1-\beta) l + \beta r$. One then has
\begin{align*} 
\partial_{Q_0}V_{rev}(0,\theta) 
=&\frac{1}{z_1}\partial_{A}G_1(A(0,\theta),0,\theta)\partial_{Q_0}A(0,\theta)+\frac{1}{z_1}\partial_{Q_0}G_1(A(0,\theta),0,\theta)\\
=& \dfrac{1-\theta^2}{z_1^2} \dfrac{(\beta - \alpha)(l-r)}{\big((1-\alpha) l + \alpha r \big)\big((1-\beta) l + \beta r \big)}.
\end{align*}
The expansion (\ref{partQ-Vrev}) then follows from that   $V_{rev}(0,\theta) =\dfrac{\theta}{z_1}  \ln \dfrac{l}{r}$.  
\end{proof}

 It follows from (\ref{partQ-Vrev}) that, for small $Q_0$, the reversal potential $V_{rev}(Q_0,0)$ is of order $O(1)$ in general (if $l\neq r$) but,  if $D_1=D_2$ (so that $\theta=0$), then the reversal potential $V_{rev}(Q_0,0) =O(Q_0)$. This is consistent with the result in Corollary \ref{VQ} and the statement followed the corollary.


\subsection{Dependence of the Reversal Potential $V_{rev}$  on $\theta$.}
 Recall that $\theta =(D_2-D_1)/{(D_2+D_1)}$ is a measurement of the difference between $D_1$ and $D_2$
 \begin{prop}\label{Vonp}  One has
$\partial_{\theta}V_{rev}(Q_0,\theta)$ has the same sign as that of  $l-r$.
\end{prop}
\begin{proof} 
Direct calculations from (\ref{VonQf}) give
$$
\begin{aligned}
\partial_{\theta}V_{rev}(Q_0,\theta) 
= & \dfrac{1}{z_1} \Big(g(S_a)-g(S_b) +  \ln \dfrac{lB(Q_0,\theta)}{rA(Q_0,\theta)}   \Big),
\end{aligned}
$$
where $g(X)$ is defined in (\ref{gg}) and is increasing in $X$ for $X>0$. 
In particular, if   $l>r$ then $g(S_a)-g(S_b)>0$. Moreover, it follows from Theorem \ref{AonQ}   that  $r < B(Q_0,\theta) < A(Q_0,\theta) < l$. The proof is thus complete for $l>r$. The case for $l<r$ is similar. 
\end{proof}

\begin{rem} \label{remVonp}
Proposition \ref{Vonp} shows how diffusion coefficients affect reversal potential and reveals a fascinating attribute that may not be completely intuitive at first glance. Indeed, 
recall  the observation in \cite{ELX15}  that, for $k=1,2$,
\begin{align}\label{JonP}
\frac{J_k}{D_k} \int_0^1\frac{1}{h(x)c_k(x)}dx=z_kV+\ln \frac{l}{r}.
\end{align}
The relation, of course, holds for the zero current condition: $J_1=J_2$ with $V=V_{rev}$. Now, if we fix the diffusion constant $D_1$ but increase $D_2$ (so $\theta$ is increasing), then { $|J_2|$} increases since  all but $\dfrac{J_2}{D_2}$ in (\ref{JonP}) are independent of $D_2$ (\cite{Liu09}), and hence, to satisfy zero current condition,    we should   increase  { $|J_1|$}. Intuitively   increasing $V_{rev}$ seems to accomplish the latter.  But this intuition agrees with Proposition \ref{Vonp} only for $l>r$ and is the exactly opposite for $l<r$. That is, for $l<r$,    Proposition \ref{Vonp} says, as $\theta$ increases, $V_{rev}(Q_0,\theta)$ decreases. This counterintuitive behavior could be explained by the fact that  $c_1(x)$ actually depends on $V_{rev}$ and reducing $V_{rev}$ could   increase { $|J_1|$}.
Unfortunately, we could not explain the behavior in physical terms and will conduct  further investigation needed to better understand the behavior. \qed
\end{rem}


\subsection{A Comparison to Goldman-Hodgkin-Katz Equation for  $V_{rev}$.}\label{Ef_of_Q}
We will first recall Goldman-Hodgkin-Katz (GHK) equation for the reversal potential $V_{rev}$  and then make a comparison with our result.
 
Based on essentially the assumption that the electric potential $\phi(x)$ is linear in $x$ (or the electric field is constant), 
Goldman (\cite{Goldman}), and    Hodgkin and Katz (\cite{HK49}) derived an equation (the GHK equation) for the reversal potential, which extends that of Nernst equation for a single ion species.  Under the assumption, the I-V (current-voltage) relation is given by
\[I = V \sum_{k=1}^n z_k^2 D_k \dfrac{r_k - l_k e^{z_kV}}{1- e^{z_kV}}.\]
For the   case where $n=2$ and $z_1 = -z_2$,   the GHK  equation for the reversal potential    is  
\begin{align}\label{N-GHK}
V_{rev}^{GHK}(\theta) = \dfrac{1}{z_1} \ln \dfrac{(1-\theta)r + (1+\theta) l}{(1-\theta)l + (1+\theta) r}.
\end{align}

 The assumption that the electric potential $\phi(x)$ is linear in $x$ is thought to probably make sense without channel structure, in particular, when $Q_0=0$. This is not correct either. In fact, when $Q_0=0$, from Corollary \ref{VQ}, the reversal potential is
\[V_{rev}(0,\theta) = \dfrac{\theta}{z_1}  \ln \dfrac{l}{r},\] 
which is different from that in (\ref{N-GHK}). In our opinion, what is more important is that our result on the reversal potential is the first for  general $Q_0\neq 0$ with different diffusion coefficient. Thus, for $n=2$ with $z_1=-z_2$, the GHK equation for reversal potential should be replaced by 
\[V_{rev}(Q_0,\theta)={\dfrac{1}{z_1} }G_1(A(Q_0,\theta),Q_0,\theta)\]
with $A(Q_0,\theta)$ being the solution of  $G_2(A,Q_0,\theta)=0$. Figure \ref{Fig-PermChargeHx} shows  comparisons between  the reversal potential $V_{rev}$ from our result with $V_{rev}^{GHK}$ as functions of $D_2/{D_1}$  with $Q_0=0$ and $Q_0=10$.

\begin{figure}[h]
	\centerline{\epsfxsize=3.0in
\epsfbox{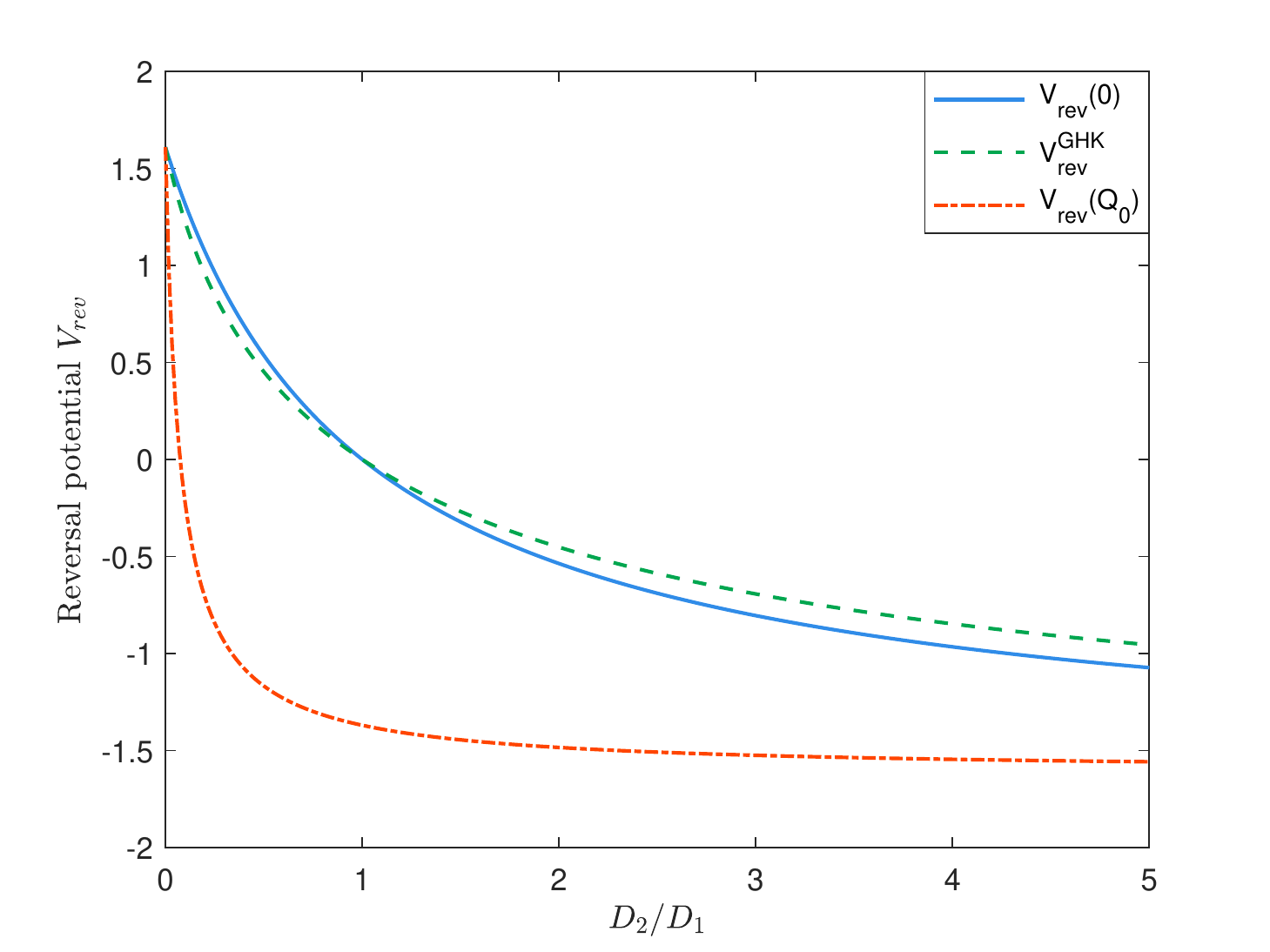}
\epsfxsize=3.0in 
\epsfbox{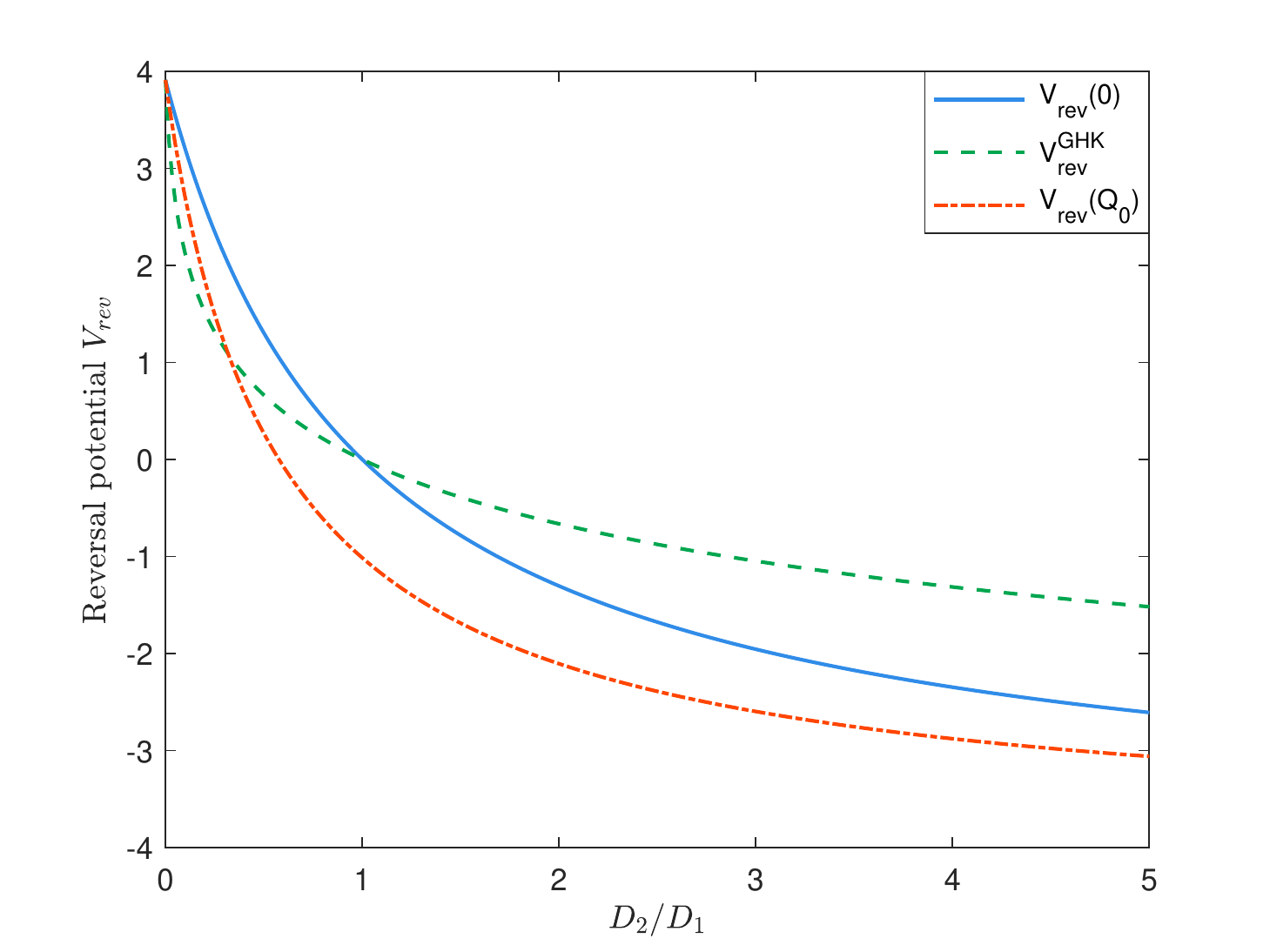} }
	\caption{\em Comparison of $V_{rev}(0)$ with $Q_0=0$, $V_{rev}^{GHK}$ and $V_{rev}(Q_0)$ with $Q_0=10$. Left panel: $l=0.2$ and  $r=1$. Right panel: $l=0.2$ and $r=10$.  }
	\label{Fig-PermChargeHx}
\end{figure}


It is very important to generalize this result to mixtures with more than two ion species.


\section{Reversal Permanent Charge $Q_{rev}(V,\theta)$.}\label{revPC}
\setcounter{equation}{0}
In view of the duality of reversal potential $V$ and the reversal permanent charge $Q^*$, we now present a general result for reversal permanent charge with a given electric potential $V$. We comment that there are differences between these two problems. On one hand, as probably expected,   reversal potentials should always exist. On the other hand, there is a simple necessary condition for the existence of the reversal permanent charge $Q_{rev}$ as discussed above.  This is indeed established below for the special case of permanent charges $Q$ in (A2).

\begin{thm}\label{thm-Qrev} For $n=2$ with $z_1=-z_2$, 
there exists a   reversal permanent charge $Q_{rev}$ if and only if 
\begin{equation}\label{Qrev-NecCond}
\big(z_1V+ \ln\dfrac{l}{r} \big) \big( z_2V + \ln\dfrac{l}{r} \big) >0.
\end{equation}
\end{thm}
\begin{proof} Since $J_k$, for $k=1,2$, has the same sign as that of $z_kV + \ln\dfrac{l}{r}$ and $z_1=-z_2$, the condition in (\ref{Qrev-NecCond}) is necessary for a zero current $I$, and hence, for the existence of a reversal permanent charge.
To show the condition is also sufficient, we set $ g_1(Q_0) := G_1(A(Q_0),Q_0,\theta)-z_1V$. From Theorem \ref{thmA(Q)} one obtains,
$$
\lim_{Q_0\to +\infty} \ln  \dfrac{S_a -Q_0}{S_b -Q_0}   = 2\lim_{Q_0\to  +\infty}\ln \dfrac{A}{B}= 2 \ln \dfrac{l}{r}, \quad \lim_{Q_0\to -\infty} \ln  \dfrac{S_a -Q_0}{S_b -Q_0}  = \lim_{Q_0\to \pm\infty}\ln\dfrac{S_a + \theta Q_0}{S_b + \theta Q_0}=0.
$$
Then from above, the equation for $G_1$ in \eqref{G}, Lemma \ref{thmA(Q)} and above one has,
\begin{equation}\label{ineqlimg1}
\begin{aligned}
 \lim_{Q_0\to +\infty}g_1(Q_0)= -z_1V+\ln\dfrac{l}{r} , \quad
\lim_{Q_0\to -\infty} g_1(Q_0) = -z_1V- \ln\dfrac{l}{r}. 
\end{aligned}
\end{equation}
The condition (\ref{Qrev-NecCond}) implies that the above values have opposite signs. By the Intermediate Value Theorem, there is at least one 
$Q_0=Q_{rev}(V,\theta)$ such that $g_1(Q_0)=0$. \end{proof}

This existence result can be viewed as a duality of Theorem \ref{thm-revpot} together with (i) and (ii) in Theorem \ref{VonQ}.
  The next result is a duality to (iii) of Theorem \ref{VonQ}, whose proof will be omitted.
\begin{thm}\label{lem-Qrev}  
For any   $(V, l, r)$ that satisfies the condition \eqref{Qrev-NecCond} one has,  
\[\lim\limits_{z_1V \to \theta \ln {l}/{r} } Q_{rev}(V,\theta)= 0\;\mbox{ and }\; \lim\limits_{z_1V \to \pm\ln {l}/{r} } Q_{rev}(V,\theta)= \pm\infty.\]
\end{thm}

Recall  we could not show but conjecture that $V_{rev}(Q_0,\theta)$ is monotone in $Q_0$ in Section \ref{revP}.  Should the conjecture be shown, $Q_{rev}(V,\theta)$ would be monotone in $V$. 

 
 \section{Conclusion.}\label{ConSec}
%
 	 
	 In this paper, we work on the classical PNP model allowing unequal  diffusion constants and for a single profile of permanent charges, to study the specific questions about reversal potentials and reversal permanent charges that are among the central issues of biological concerns.

	 A major challenge for  study  properties of ionic flow through ion channel lies in the fact that there are many specific physical parameters, including the boundary concentrations and transmembrane electric potential, permanent charge (the value $Q_0$ and the characteristic distribution parameters $\alpha$ and $\beta$), diffusion coefficients that all are relevant and interact with each other nonlinearly: {\em different regions of the large dimensional parameter space are associated with different properties}. Furthermore all present experimental measurements about ionic flow are of input-output type; that is, the internal dynamics within the channel cannot be measured with the current technology. It is thus extremely hard to extract coherent properties or to formulate specific characteristic quantities from the experimental measurements. 
Without knowing what to simulate among the potentially rich behavior presented by ion channel problems, it is also hard for numerical simulations to conduct any systematic studies.  

It is not expected that the abstract theory of singular perturbations could provide concrete results in general. For the PNP problem, the geometric singular perturbation approach applied in this paper was developed in \cite{EL07, Liu09} that relies on (i) the geometric singular perturbation theory based the advance of nonlinear dynamical system theory and (ii) two special structures of the problem -- one for the limiting fast subsystem and the other for the limiting slow subsystem. As a result, the zeroth order terms for the asymptotic solution are determined by {\em a governing system} -- a system of algebraic equations that involves {\em all parameters} of the boundary value problem. At least for simple setups, this framework led to discovery of rich effect of permanent charge on cation flux and anion flux (\cite{JLZ15,ZEL}), revealing of a mechanism of declining phenomenon -- increasing of the transmembrane electrochemical potential of an ion species in a particular way leads to decreasing of the ionic  flux (\cite{ZEL}),  formulate critical values for ionic flows (\cite{ELX15,JL12, LLYZ,  Liu18}).  For the reversal potential and reversal permanent charge problems studied in this paper, the governing system consists of (\ref{Matching}) and (\ref{Matching2}). The crucial step of the analysis in this paper is our reduction of the system (\ref{Matching}) and (\ref{Matching2})
to the single equation $G_2(A,Q_0,\theta)=0$ in (\ref{G1G2Sys}) for the reversal potential and to the system of two equations in (\ref{G1G2Sys}) for the reversal permanent charge -- both involve the other physical parameters.  We are then able to examine a number of properties about the dependence of the reversal potential on permanent charge, diffusion coefficients, etc.   	We like to point out that   the differential equation system (\ref{PNP}) is a stiff problem to solve numerically but  the system (\ref{G1G2Sys}) is a regular problem for numerical simulations. So the latter is much easier to solve numerically.
 	 	
 	On the other hand, one can only expect that analysis could provide detailed and concrete information in special cases or for a certain limit parameter values. Numerical simulation can take the advantage of knowing what are important from such an analysis and extend the setup to more physical meaningful ranges of parameters.  This has been done in many works, for example, in \cite{LTZ}, numerical simulations were conducted to compute some critical electric potentials, which were guided by the analytical result in \cite{JL12} that provides the defining properties of the critical electric potentials and explicit formulas for the critical values in terms of boundary concentrations etc. in special cases.  
	
	It is our plan to continue the study to extend the capability of the analysis in this work and to examine the topics numerically using more sophisticated models.
	



 \section{Appendix:  Proof of Proposition \ref{rSys}.}\label{Appendix}
 \setcounter{equation}{0}

We consider a special case where $z_1=-z_2$. Set $c_1^{[1]}c_2^{[1]}=A^2$ and $c_1^{[2]}c_2^{[2]}=B^2$. We will use the notion $l$, $r$, $Q_0$,    $ \alpha$, $\beta$, $S_a$, $S_b$ and $N$ introduced in (\ref{lrp}) and (\ref{abp}). 

With above terms, from \eqref{Matching2} we get
\begin{equation}\label{c-albr-m}
\begin{aligned}
&c_1^{[1,-]} = c_2^{[1,-]}=A,\quad c_1^{[2,+]}= c_2^{[2,+]}= B, \quad c_1^{[1,+]}= \frac{S_a - Q_0}{z_1}, \quad c_1^{[2,-]}= \frac{S_b - Q_0}{z_1}.
\end{aligned}
\end{equation}
 From the third and fourth equations in \eqref{Matching} one has,
\begin{equation}\label{phi-}
 \phi^{[1]} - \phi^{[1,+]} = \dfrac{1}{z_1}\ln \dfrac{S_a -Q_0}{z_1c_1^{[1]}},\quad \phi^{[2]} - \phi^{[2,-]} = \dfrac{1}{z_1} \ln \dfrac{S_b -Q_0}{z_1c_1^{[2]}}.
 \end{equation}

Then, from the first two equations of \eqref{Matching}, \eqref{Matching2} and \eqref{phi_al,br}  give
\begin{equation}\label{phi-ab}
\begin{aligned}
& \phi^{[1]}   = V + \dfrac{2D_2}{\big(D_1 + D_2\big)z_1} \ln (z_1A) + \frac{(D_1 - D_2)}{\big(D_1 + D_2\big)z_1}\ln (z_1 l) - \dfrac{1}{z_1} \ln (z_1 c_1^{[1]}) ,\\
& \phi^{[2]} = \dfrac{2D_2}{\big(D_1 + D_2\big)z_1} \ln (z_1B) + \frac{(D_1 - D_2)}{\big(D_1 + D_2\big)z_1}\ln (z_1 r) - \dfrac{1}{z_1} \ln (z_1 c_1^{[2]}).
\end{aligned}
\end{equation}
The rest of system \eqref{Matching} becomes,
\begin{equation}\label{Redu_Matching}
\begin{aligned}
c_2^{[1]}=&c_1^{[1]} e^{2z_1(\phi^{[1]}-\phi^{[1,-]})},\quad c_2^{[2]}=c_1^{[2]} e^{2z_1(\phi^{[2]}-\phi^{[2,+]})},\\
 z_1A =&S_a + Q_0 \ln \dfrac{S_a- Q_0}{z_1c_1^{[1]}},\quad
 z_1B =S_b + Q_0 \ln \dfrac{S_b- Q_0}{z_1c_1^{[2]}},\\
\frac{J_1}{D_1 D_2} =& -\dfrac{2(A- l)}{(D_1 + D_2)\alpha H(1)} = -\dfrac{2(r- B)}{(D_1 + D_2)(1-\beta) H(1)}\\
 =& -2\dfrac{(c_1^{[2,-]}- c_1^{[1,+]}) - (\phi^{[2,-]} - \phi^{[1,+]})Q_0}{(D_1 + D_2 )(\beta - \alpha)H(1)},\\
\dfrac{D_2 - D_1}{D_1 D_2} z_1J_1y^* = &\phi^{[1]} - \phi^{[2]} + \frac{1}{z_1} \ln\dfrac{c_1^{[1]}\big(S_b -Q_0 \big)}{c_1^{[2]} \big(S_a -Q_0 \big)}\\
 S_b- Q_0 =& e^{-\frac{D_1 + D_2}{D_1 D_2}z_1^2J_1y^*}\big(S_a - Q_0 \big) -\dfrac{2D_2Q_0}{D_1 + D_2 } \big(1-e^{-\frac{D_1 + D_2}{D_1 D_2}z_1^2J_1y^*} \big).
\end{aligned}
\end{equation}

From third and fourth equations in \eqref{Redu_Matching},
\begin{equation}\label{c1-ab}
\begin{aligned}
c_1^{[1]} = & \dfrac{S_a -Q_0}{z_1}\exp\Big\{\dfrac{S_a - z_1A}{Q_0}\Big\},\quad
c_1^{[2]} =  \dfrac{S_b -Q_0}{z_1} \exp\Big\{\dfrac{S_b - z_1B}{Q_0}\Big\}.
\end{aligned}
\end{equation}
The equations \eqref{phi-ab} and \eqref{c1-ab} give
\begin{equation}
\begin{aligned}
\phi^{[1]} =& V + \dfrac{2D_2}{\big(D_1 + D_2\big)z_1} \ln (z_1A) + \dfrac{D_1 - D_2}{(D_1 + D_2)z_1} \ln(z_1 l) - \dfrac{1}{z_1}\ln(S_a -Q_0)- \dfrac{S_a- z_1A}{z_1Q_0},\\
\phi^{[2]} =& \dfrac{2D_2}{\big(D_1 + D_2\big)z_1} \ln (z_1B) + \dfrac{D_1 - D_2}{(D_1 + D_2)z_1} \ln(z_1 r) - \dfrac{1}{z_1}\ln (S_b -Q_0 )- \dfrac{S_b - z_1B}{z_1Q_0}.
\end{aligned}
\end{equation}
Thus
\begin{equation}\label{phi(b-a)}
\begin{aligned}
\phi^{[2]} - \phi^{[1]}=& -V + \dfrac{2D_2}{\big(D_1 + D_2\big)z_1} \ln \dfrac{B}{A} - \dfrac{D_1 - D_2}{(D_1 + D_2)z_1} \ln \dfrac{l}{r}\\
& - \dfrac{1}{z_1}\ln \dfrac{S_b -Q_0}{S_a -Q_0}+ \dfrac{S_a - S_b + z_1(B-A)}{z_1Q_0}.
\end{aligned}
\end{equation}
Now, the equation  \eqref{phi-} and $y^*$ equation in sixth line of \eqref{Redu_Matching} give,
\[\phi^{[2,-]}- \phi^{[1,+]} = \phi^{[2]} - \phi^{[1]} - \frac{1}{z_1} \ln\dfrac{c_1^{[1]} \big(S_b -Q_0 \big)}{c_1^{[2]} \big(S_a -Q_0 \big)}.\]
But, from third and fourth equations of \eqref{Redu_Matching},
\[\ln\dfrac{c_1^{[1]} \big(S_b -Q_0 \big)}{c_1^{[2]} \big(S_a -Q_0 \big)}= \dfrac{1}{Q_0}\big(S_a -S_b+z_1(B-A) \big),\]
and hence,
\begin{equation}\label{phi-ambm2}
\begin{aligned}
\phi^{[2,-]}- \phi^{[1,+]} =& -V + \dfrac{2D_2}{\big(D_1 + D_2\big)z_1} \ln \dfrac{B}{A} - \dfrac{D_1 - D_2}{(D_1 + D_2)z_1} \ln \dfrac{l}{r} - \dfrac{1}{z_1}\ln \dfrac{S_b -Q_0}{S_a -Q_0} .
\end{aligned}
\end{equation}
Furthermore, it follows from above that,
\begin{equation}\label{J1y*-1}
\begin{aligned}
\phi^{[2,-]}- \phi^{[1,+]} =& \phi^{[2]} - \phi^{[1]} -  \dfrac{1}{z_1Q_0}\big(S_a -S_b+z_1(B-A) \big)
 =\dfrac{D_1 - D_2}{D_1 D_2} z_1J_1y^*.
\end{aligned}
\end{equation}
Thus, $J_1$ equations in \eqref{Redu_Matching}, with equations in (\ref{J1y*-1}) and (\ref{c-albr-m}) give,
\begin{equation}\label{J}
\begin{aligned}
\dfrac{J_1}{D_1 D_2} =& -\dfrac{2(A- l)}{(D_1 + D_2)\alpha H(1)} = -\dfrac{2(r- B)}{(D_1 + D_2)(1-\beta) H(1)}\\
 =& -2\dfrac{B-A-Q_0(\phi^{[2]} - \phi^{[1]}) }{(D_1 + D_2)(\beta - \alpha)H(1)}.
\end{aligned}
\end{equation}

\noindent Now, from the equations in \eqref{J},
\begin{equation}\label{Bphi}
B = \dfrac{1-\beta}{\alpha}(l-A)+r,\quad \phi^{[2]} - \phi^{[1]} =- \frac{A-l + \alpha (l-r)}{\alpha Q_0}.
\end{equation}
Thus, the equations in  \eqref{J1y*-1} and \eqref{Bphi} give,
\begin{equation}\label{J1y*-2}
\begin{aligned}
  \dfrac{J_1y^*}{D_1 D_2} =&  \dfrac{1}{z_1^2\big(D_2 - D_1\big)Q_0}N(A,Q_0),
\end{aligned}
\end{equation}
where $N =N(A,Q_0)=\displaystyle{ \dfrac{\beta - \alpha}{\alpha}z_1(A-l)  +S_a -S_b.
}$ is   defined in \eqref{abp}. 
On the other hand, from \eqref{phi(b-a)} and \eqref{Bphi} we obtain an equation in terms of $A$ and $Q_0$,
$$
\begin{aligned}
	\frac{(\beta - \alpha) (A-l)}{\alpha Q_0}  = & V - \dfrac{2D_2}{\big(D_1 + D_2\big)z_1} \ln \dfrac{B}{A} + \dfrac{D_1 - D_2}{(D_1 + D_2)z_1} \ln \dfrac{l}{r}	 + \dfrac{1}{z_1}\ln \dfrac{S_b -Q_0}{S_a -Q_0}- \dfrac{S_a -S_b }{z_1Q_0}.
\end{aligned}
$$
Now, it follows from above equation and the expression for $N(A,Q_0)$ that,
\begin{equation}\label{Apn-G1}
\dfrac{N}{Q_0} -  z_1V - \dfrac{2D_2}{D_1 + D_2} \ln \dfrac{A}{B} - \dfrac{D_1 - D_2}{D_1 + D_2}  \ln \dfrac{l}{r} + \ln \dfrac{S_a -Q_0}{S_b -Q_0} =0.
\end{equation}
Substituting \eqref{J1y*-2} into the last equation of \eqref{Redu_Matching} we get the other equation for $A$ and $Q_0$,
\[ \sqrt{Q_0^2+z_1^2B^2} - Q_0 = e^{\frac{(D_1 + D_2)}{(D_1 - D_2)Q_0}N}\big(\sqrt{Q_0^2+z_1^2A^2} - Q_0 \big)
  -\dfrac{2D_2Q_0}{D_1 + D_2} \big(1-e^{\frac{(D_1 + D_2)}{(D_1 - D_2)Q_0}N} \big),\]
that is equivalent to
$$
 \frac{(D_2 - D_1)Q_0}{(D_1 + D_2)}\ln \dfrac{S_a + \frac{(D_2 - D_1)Q_0}{(D_1 + D_2)}}{S_b + \frac{(D_2 - D_1)Q_0}{(D_1 + D_2)}} -N=0.
$$
This equation is $G_2=0$ in Proposition \ref{rSys}. Also, adding ${G_2}/{Q_0} $ to  equation \eqref{Apn-G1} one obtains
  $G_1=z_1V$ in Proposition \ref{rSys}. This completes the proof of Proposition \ref{rSys}.
\bigskip

\noindent 
{\bf Acknowledgement.}  The authors thank the anonymous referees for careful reviews of and invaluable comments on the original manuscript that significantly improved the paper.  Hamid Mofidi thanks Mingji Zhang, and both authors thank Bob Eisenberg for helpful discussions.    
Weishi  Liu's research is partially supported by the Simons Foundation 
Mathematics and Physical Sciences-Collaboration Grants for Mathematicians \#581822. 

 \small
 
  \bibliographystyle{plain}

\begin{thebibliography}{999} 

 \bibitem {AEL} N. Abaid, R. S. Eisenberg, and W. Liu,
Asymptotic expansions of I-V relations via a Poisson-Nernst-Planck system.
{\em SIAM J. Appl. Dyn. Syst.} {\bf  7} (2008), 1507-1526.

{
\bibitem {BK18} B. Balu and A. Khair, Role of Stefan-Maxwell fluxes in the dynamics of concentrated electrolytes. 
{\em Soft Matter.} {\bf 14} (2018),   8267-8275.
}

{
	\bibitem {BKSA09} M.Z. Bazant, M.S. Kilic, B.D. Storey, and A. Ajdari, Towards an understanding of induced charge electrokinetics at large applied voltages in concentrated solutions. 
	{\em Adv. Coll. Interf. Sci.} {\bf 152} (2009),   48-88.
}      

\bibitem {Bar} V. Barcilon,  
Ion flow through narrow membrane channels: Part I.
{\em SIAM J. Appl. Math.} {\bf 52} (1992),   1391-1404.

 \bibitem {BCE} V. Barcilon,  D.-P. Chen,   and R. S. Eisenberg,   
 Ion flow through narrow membrane channels: Part II.
{\em SIAM J. Appl. Math.} {\bf 52} (1992),  1405-1425.
 
\bibitem {BCEJ} V. Barcilon,  D.-P. Chen,    R. S. Eisenberg,   and J. W. Jerome,  
Qualitative properties of steady-state Poisson-Nernst-Planck systems:
Perturbation and simulation study.
{\em SIAM J. Appl. Math.} {\bf 57} (1997),  631-648.


%
  
\bibitem{BNVHEG} D. Boda, W. Nonner, M. Valisko, D. Henderson, B. Eisenberg, and D. Gillespie,
Steric selectivity in Na channels arising from protein polarization and mobile side
chains. 
{\em Biophys. J.} {\bf 93} (2007), 1960-1980.
 
%
%


\bibitem{CE} D. P. Chen  and R.S. Eisenberg, 
Charges, currents and potentials in ionic channels
of one conformation. 
{\em Biophys. J.} {\bf 64} (1993),   1405-1421.

\bibitem{CEJS95} D. Chen, R. Eisenberg, J. Jerome and C. Shu, Hydrodynamic model of temperature change in open ionic channels. {\em Biophysical J.} {\bf 69} (1995), 2304-2322.

\bibitem{ChK} S. Chung and S. Kuyucak, Predicting channel
function from channel structure using Brownian dynamics simulations.
{\em Clin. Exp. Pharmacol Physiol.} {\bf 28} (2001),    89-94.
  
%
%

\bibitem{Eis} B. Eisenberg,  
 Ion Channels as Devices.
{\em J. Comp. Electro.} {\bf 2} (2003),   245-249.
 
\bibitem{Eis00} B. Eisenberg, 
Crowded charges in ion channels. 
{\em Advances in Chemical Physics} (ed. S. A. Rice) (2011), 77-223, 
John Wiley and Sons, Inc. New York.

 

%
%
%
\bibitem{EHL10} B. Eisenberg, Y. Hyon, and C. Liu,
Energy variational analysis of ions in water and channels:
Field theory for primitive models of complex ionic fluids.
{\em J. Chem. Phys.} {\bf 133} (2010),  104104(1-23).

 

\bibitem{EL07} B. Eisenberg and W. Liu,
Poisson-Nernst-Planck systems for ion channels with permanent charges. 
{\em SIAM J. Math. Anal.} {\bf 38} (2007),   1932-1966.

{
 \bibitem{EL17} B. Eisenberg and W. Liu,
Relative dielectric constants and selectivity ratios in open ionic channels. 
{\em Mol. Based Math. Biol.} {\bf 5} (2017),   125-137.
}

\bibitem{ELX15} B. Eisenberg, W. Liu, and H. Xu, 
Reversal permanent charge and reversal
potential: case studies via classical Poisson-Nernst-Planck models. 
{\em Nonlinearity} {\bf 28} (2015), 103-128.

 
\bibitem{Fen} N. Fenichel,  
Geometric singular perturbation theory for ordinary differential equations.
{\em J. Differential Equations} {\bf 31} (1979),   53-98.

 {
 \bibitem{GE02} D. Gillespie,  and R. S. Eisenberg,  
 Physical descriptions of experimental selectivity measurements
 in ion channels.
 {\em Eur Biophys J.} {\bf 31} (2002),   454-466.
}

\bibitem{GNE} D. Gillespie,  W. Nonner,  and R. S. Eisenberg,  
Coupling Poisson-Nernst-Planck and density functional theory
to calculate ion flux.
{\em J. Phys.: Condens. Matter} {\bf 14} (2002),   12129-12145.


\bibitem{Goldman} D. E. Goldman,  Potential, impedance, and rectification in membranes.
{\em J. Gen. Physiol.} {\bf 27} (1943), 37-60.

{
\bibitem{HBRR18} S. M. H. Hashemi Amrei, S. C. Bukosky, S. P. Rader, W. D.
Ristenpart, and G. H. Miller, Oscillating Electric Fields in
Liquids Create a Long-Range Steady Field. 
{\em  Phys. Rev. Lett.} {\bf 121} 185504(2018).	
}


\bibitem{Hil01} B. Hille,   {\em Ion Channels of Excitable Membranes (Third Edition).}
 Sinauer Associates, Inc., Sunderland, Massachusetts, USA 2001.

\bibitem{Hille89} B. Hille,  Transport Across Cell Membranes: Carrier Mechanisms, Chapter 2.
           Textbook of Physiology (ed. H. D. Patton, A. F. Fuchs, B. Hille, A. M. Scher and R. D. Steiner).
           Philadelphia, Saunders {\bf 1} (1989), 24-47.
%
%
 \bibitem{HEL10} Y. Hyon, B. Eisenberg,  and C. Liu,
A mathematical model for the hard sphere repulsion in ionic solutions.
 {\em Commun. Math. Sci.} {\bf 9} (2010),  459-475.
%

 \bibitem{HFEL12} Y. Hyon, J. Fonseca, B. Eisenberg,  and C. Liu,
 Energy variational approach to study charge inversion (layering) near  charged walls.
 {\em  Discrete Cont. Dyn. Syst. B} {\bf 17} (2012),  2725-2743.

\bibitem{HPS} M. Hirsch, C. Pugh, and M. Shub,
 {\em Invariant Manifolds.}
 Lecture Notes in Math. {\bf 583}, Springer-Verlag, New York, 1976.
 
 \bibitem{HHK49} A. L.  Hodgkin,  A. Huxley, and B. Katz,
 Ionic currents underlying activity in the giant axon of the squid.
{\em Arch. Sci. Physiol.} {\bf 3} (1949), 129-150.

\bibitem{HK49} A. L. Hodgkin and B. Katz,
  The effect of sodium ions on the electrical activity of the giant axon of the squid.
  {\em J. Physiol.} {\bf 108} (1949),  37-77.
  
 \bibitem{H51} A. L. Hodgkin,
 The ionic basis of electrical activity in nerve and muscle.
 {\em Biological Rev.} {\bf 26} (1951), 339-409.


\bibitem{HH52a} A. L. Hodgkin and A. F. Huxley,
Currents carried by sodium and potassium ions through the membrane of the giant axon of {\em Loligo}.
{\em J. Physol.} {\bf 116} (1952), 449-472.

\bibitem{HH52b} A. L. Hodgkin and A. F. Huxley,
The components of membrane conductance in the giant axon of {\em Loligo}.
{\em J Physiol.} {\bf 116} (1952), 473-496.

\bibitem{HH52c} A. L. Hodgkin and A. F. Huxley,
A quantitative description of membrane current and its application to conduction and excitation in nerve.
{\em J. Physiol.} {\bf 117} (1952), 500-544.

\bibitem{HK55} A. L. Hodgkin and R. D. Keynes,
The potassium permeability of a giant nerve fibre.
{\em J. Physiol.} {\bf 128} (1955), 61-88.

\bibitem{HCE} U. Hollerbach,   D.-P. Chen, and R.S. Eisenberg,
Two- and Three-Dimensional Poisson-Nernst-Planck Simulations of
Current Flow through Gramicidin-A.
{\em J.   Comp. Science} {\bf 16} (2001),  373-409. 

\bibitem{IR} W. Im  and B. Roux, 
Ion permeation and selectivity of
OmpF porin: a theoretical study based on molecular dynamics,
Brownian dynamics, and continuum electrodiffusion theory.
{\em J. Mol. Biol.} {\bf 322} (2002),    851-869.
 
 \bibitem{JEL17} S. Ji, B. Eisenberg, and W. Liu,
 Flux Ratios and Channel Structures.
 {\em J. Dynam. Differential Equations} (2017). https://doi.org/10.1007/s10884-017-9607-1.
 
\bibitem{JL12} S. Ji and W. Liu,
Poisson-Nernst-Planck systems for ion flow with density functional
theory for  hard-sphere potential: I-V relations and critical potentials. Part I: Analysis.  
{\em J. Dynam. Differential Equations} {\bf 24} (2012), 955-983. 
  
 \bibitem{JLZ15} S. Ji,  W. Liu, and M. Zhang,
Effects of (small) permanent charge and channel geometry on ionic flows
via classical Poisson-Nernst-Planck models,
{\em SIAM J. Appl. Math.} {\bf 75} (2015), 114-135.

%
%
%
%
%
%
 
\bibitem{LLYZ} G. Lin, W. Liu, Y. Yi, and M. Zhang,  
Poisson-Nernst-Planck systems for ion flow with a local hard-sphere potential for ion size effects. 
{\em SIAM J. Appl. Dyn. Syst.} {\bf 12} (2013),  1613-1648.

 
\bibitem{Liu05} W. Liu, 
Geometric singular perturbation approach to steady-state Poisson-Nernst-Planck systems.
{\em SIAM J. Appl. Math.} {\bf 65} (2005),   754-766.

\bibitem{Liu09} W. Liu, 
One-dimensional steady-state Poisson-Nernst-Planck systems for ion channels with multiple ion species. 
{\em J. Differential Equations} {\bf 246} (2009),    428-451.

\bibitem{LW} W. Liu  and B. Wang,  
Poisson-Nernst-Planck systems for narrow tubular-like membrane channels. 
{\em J. Dynam. Differential Equations} {\bf 22} (2010),    413-437.


\bibitem{LTZ} W. Liu, X. Tu, and M. Zhang,
Poisson-Nernst-Planck systems for ion flow with density functional theory for hard-sphere potential: 
I-V relations and critical potentials. Part II: Numerics. 
{\em J. Dynam. Differential Equations} {\bf 24} (2012),   985-1004.


\bibitem{LX15} W. Liu and H. Xu, 
A complete analysis of  a classical Poisson-Nernst-Planck model for ionic flow.
{\em  J. Differential Equations} {\bf 258} (2015), 1192-1228.
 \bibitem{Liu18} W. Liu, A flux ratio and a universal property of
permanent charges effects on fluxes. 
{\em Comput. Math. Biophys.} {\bf 6} (2018), 28-40.
 
\bibitem{MEL20} H. Mofidi, B. Eisenberg and W. Liu, 
Effects of Diffusion Coefficients and Permanent Charge on Reversal Potentials in Ionic Channels.
{\em  Entropy} {\bf 22} (2020), 325(1-23).


\bibitem{Mott39} N. F. Mott, The theory of crystal rectifiers. 
{\em Proc. Roy. Soc. A} {\bf 171} (1939), 27-38.

%
%
%
%

\bibitem{NE} W. Nonner  and R. S.  Eisenberg,  
Ion permeation and glutamate residues linked by Poisson-Nernst-Planck theory in L-type Calcium channels.
{\em Biophys. J.} {\bf 75} (1998),   1287-1305.


 \bibitem{NCE} W. Nonner, L. Catacuzzeno,  and B.  Eisenberg,  
Binding and selectivity in L-type Calcium channels: 
A mean spherical approximation.
{\em Biophys. J.} {\bf 79} (2000),   1976-1992.

\bibitem{PJ} J.-K. Park  and J. W. Jerome,   
 Qualitative properties of steady-state Poisson-Nernst-Planck systems: Mathematical study.
{\em SIAM J. Appl. Math.} {\bf 57} (1997),  609-630.


	\bibitem{RC13} A. Rudin and P. Choi,
	{\em The Elements of Polymer Science and Engineering (Third Edition).}
	10.1016/B978-0-12-382178-2.00003-1 (89-148), (2013).
	

\bibitem{SR81}
V. Sasidhar and E. Ruckenstein, { Electrolyte osmosis through capillaries.}{\em J. Colloid Interface Sci.} {\bf 82} (1981), 439-457.





  
%
%
%
\bibitem{SNE01} Z. Schuss, B. Nadler,  and R. S.  Eisenberg,
Derivation of Poisson and Nernst-Planck equations in a bath and
channel from a molecular model.
{\em Phys. Rev. E} {\bf 64} (2001),    1-14.
%
%
%
  
  \bibitem{SL18} L. Sun and W. Liu, Non-localness of excess potentials and boundary value problems of Poisson-Nernst-Planck systems for ionic flow: A Case Study.
  {\em J. Dynam. Differential Equations} {\bf 30} (2018), 779-797.
  
  \bibitem{ZEL} L. Zhang, B. Eisenberg, and W. Liu,
  An effect of large permanent charge: Decreasing flux
with increasing transmembrane potential.
{\em Eur. Phys. J. Special Topics} {\bf 227} (2019), 2575-2601.
 \end{thebibliography}

\end{document}